\documentclass[11pt, letterpaper]{amsart}

\usepackage{lipsum}

\usepackage{graphicx}
\usepackage{array}
\usepackage{amssymb}
\usepackage{hyperref}
\usepackage{amsthm}
\usepackage{amsfonts}
\usepackage{amsmath}
\usepackage{bm}
\usepackage{mathrsfs}
\usepackage[all]{xy}
\usepackage{color}
\usepackage{subfigure}
\usepackage{tikz, tikz-cd}
\usepackage{enumerate}
\usepackage{anysize}
\usepackage{amscd}
\usepackage[letterpaper]{geometry}
\usepackage{multirow}
\usepackage{array}
\usepackage{booktabs}

\newtheorem{theorem}{Theorem}[section]
\newtheorem{proposition}[theorem]{Proposition}

\newtheorem{lemma}[theorem]{Lemma}

\newtheorem{corollary}[theorem]{Corollary}

\theoremstyle{remark}

\theoremstyle{definition}
\newtheorem{definition}[theorem]{Definition}
\newtheorem{remark}[theorem]{Remark}

\numberwithin{equation}{section}
\numberwithin{figure}{section}
\numberwithin{table}{section}

\newcommand{\bx}{\bm{x}}
\newcommand{\dC}{\mathbb{C}}
\newcommand{\dD}{\mathbb{D}}
\newcommand{\dN}{\mathbb{N}}
\newcommand{\dR}{\mathbb{R}}

\renewcommand{\i}{\sqrt{-1}}

\newcommand{\dZ}{\mathbb{Z}}

\newcommand{\bo}{\bm{0}}
\newcommand{\by}{\bm{y}}
\newcommand{\M}{\mathcal{M}}

\newcommand{\fA}{\mathfrak{A}}
\newcommand{\fB}{\mathfrak{B}}
\newcommand{\fd}{\mathfrak{d}}
\newcommand{\fr}{\mathfrak{r}}
\newcommand{\fs}{\mathfrak{s}}
\newcommand{\ft}{\mathfrak{t}}
\newcommand{\fM}{\mathfrak{M}}
\newcommand{\fN}{\mathfrak{N}}

\newcommand{\RR}{\mathbb{R}}
\newcommand{\ZZ}{\mathbb{Z}}
\newcommand{\PP}{\mathbb{P}}
\newcommand{\p}{\partial}

\newcommand{\cK}{\mathcal{K}}
\newcommand{\cM}{\mathcal{M}}

\newcommand{\cR}{\mathcal{R}}
\newcommand{\cS}{\mathcal{S}}

\newcommand{\sq}{\sqrt{-1}}

\newcommand{\cC}{\mathcal{C}}

\newcommand{\CC}{\mathbb{C}}

\newcommand{\ux}{{\underline{\bm{x}}}}

\DeclareMathOperator{\ALG}{ALG}

\DeclareMathOperator{\diam}{Diam}
\DeclareMathOperator{\dvol}{dvol}

\DeclareMathOperator{\FF}{flat}

\DeclareMathOperator{\Id}{Id}
\DeclareMathOperator{\I}{I}
\DeclareMathOperator{\II}{II}
\DeclareMathOperator{\III}{III}
\DeclareMathOperator{\IV}{IV}
\DeclareMathOperator{\V}{V}

\DeclareMathOperator{\InjRad}{InjRad}

\DeclareMathOperator{\K3}{K3}

\DeclareMathOperator{\SF}{sf}

\DeclareMathOperator{\Ima}{Im}

\DeclareMathOperator{\Rea}{Re}

\DeclareMathOperator{\Rm}{Rm}

\DeclareMathOperator{\TF}{tf}
\DeclareMathOperator{\Tr}{Tr}
\DeclareMathOperator{\Vol}{Vol}

\DeclareMathOperator{\Nil}{Nil}

\DeclareMathOperator{\er}{err}

\allowdisplaybreaks[3]

\begin{document}

 \title[Torelli-type theorems for gravitational instantons]{Torelli-type theorems for gravitational instantons with quadratic volume growth}

 \author{Gao Chen} \thanks{The first named author is supported by the Project of Stable Support for Youth Team in Basic Research Field, Chinese Academy of Sciences, YSBR-001. The second named author is supported by NSF Grants DMS-1811096 and DMS-2105478. The third named author is supported by NSF Grant DMS-1906265. }
  \address{Institute of Geometry and Physics, University of Science and Technology of China, Shanghai, China, 201315}
 \email{chengao1@ustc.edu.cn}
 \author{Jeff Viaclovsky}
 \address{Department of Mathematics, University of California, Irvine, CA 92697}
 \email{jviaclov@uci.edu}
 \author{Ruobing Zhang}
\address{Department of Mathematics, Princeton University, Princeton, NJ, 08544}
\email{ruobingz@princeton.edu}

\begin{abstract}
We prove Torelli-type uniqueness theorems for both ALG$^*$ gravitational instantons and ALG gravitational instantons which are of order $2$. That is, the periods uniquely characterize these types of gravitational instantons up to diffeomorphism. We define a period mapping $\mathscr{P}$, which we show is surjective in the ALG cases, and open in the ALG$^*$ cases. 
We also construct some new degenerations of hyperk\"ahler metrics on the K3 surface which exhibit bubbling of ALG$^*$ gravitational instantons.
\end{abstract}

\date{}

\maketitle


\section{Introduction}

We begin with the following definitions.
\begin{definition}
\label{d:HK}
A hyperk\"ahler $4$-manifold $(X, g, I, J, K)$ is a Riemannian $4$-manifold $(X,g)$ with a triple of K\"ahler structures $(g, I), (g, J), (g, K)$ such that $IJ=K$. 
\end{definition}

We denote by $\bm{\omega} = (\omega_1, \omega_2, \omega_3)$ the K\"ahler forms associated to $I, J, K$, respectively. It is easy to see that $\omega_i$ satisfies
\begin{align}
\label{e:HKtriple}
\omega_i \wedge \omega_j = 2 \delta_{ij} \dvol_g, 
\end{align}
where $\dvol_g$ is the Riemannian volume element. Conversely, any triple of symplectic forms $\omega_i$ satisfying \eqref{e:HKtriple} determines a hyperk\"ahler structure if we replace $\omega_3$ by $-\omega_3$ if necessary.

\begin{definition}
A \textit{gravitational instanton} $(X,g, \bm{\omega})$ is a non-compact complete non-flat hyperk\"ahler $4$-manifold $X$ such that
$|\Rm_{g}|\in L^2(X)$.\end{definition}

If $X$ is a {\textit{compact}} non-flat hyperk\"ahler $4$-manifold, then it must be the K3 surface; see \cite{KodairaK3}. If $X$ is a gravitational instanton, there are many known types of asymptotic geometry of $X$ near infinity: ALE, ALF-A$_k$, ALF-D$_k$, ALG, ALH, ALG$^*$, ALH$^*$. We refer the reader to \cite{BKN, CCI, CCII, CCIII, Hein, Kronheimer, Minerbe10} and references therein for more background on gravitational instantons. 

There is a well-known Torelli Theorem for hyperk\"ahler metrics on the K3 surface, and one may ask if there is an analogue for gravitational instantons. This is known to hold in several cases: 
such a Torelli-type theorem was proved by \cite{Kronheimer} in ALE case, by \cite{minerbe} in ALF-$A_k$ case, by \cite{CCII} in ALF-$D_k$ case and by \cite{CCIII} in ALH case. In this paper, we are interested in an analogous result assuming that the metric is of type ALG or ALG$^*$.  In the ALG case, it was observed in \cite{CCIII} that the natural period map may not be injective, and a modified version of the Torelli theorem was conjectured there. In this paper, we prove the uniqueness part of this conjecture, which gives the Torelli uniqueness in the ALG case; see Theorem~\ref{t:ALG-uniqueness}. We also prove a Torelli-type uniqueness theorem in the ALG$^*$ case; see Theorem~\ref{t:ALGstar2}.
We note that recently, a Torelli-type uniqueness Theorem in the ALH$^*$ case was proved; see \cite{CJL3}. We will also define a refined period mapping $\mathscr{P}$ in both the ALG and ALG$^*$ cases, which we will show to be surjective in the ALG cases, and open in the ALG$^*$ cases; see Theorem~\ref{t:ALGperiod} and Theorem~\ref{t:ALGstarperiod}. 

Previously, gravitational instantons of type ALE, ALF, ALG, ALH, ALH$^*$ have been shown to bubble off of the K3 surface; see \cite{Donaldson-kummer, Lebrun-Singer, Foscolo, CVZ, CCIII, HSVZ}. In this paper, we also show that there exist families of Ricci-flat hyperk\"ahler metrics on the K3 surface which have ALG$^*$ gravitational instantons occuring as bubbles; see Theorem~\ref{t:K3}. These are the first known examples of this type of degeneration.  
These example are produced via a gluing theorem which is actually the crucial tool in proving the aforementioned Torelli uniqueness in the ALG$^*$ case.

\subsection{ALG gravitational instantons}
\label{ss:ALG}

For background of analysis on ALG gravitational instantons, related classification results, and relations to moduli spaces of monopoles and Higgs bundles, 
we refer the readers to \cite{BB, BM, CCIII, CherkisKapustinALG, FMSW, Hein, HHM, Mazzeo} and also the references therein. 

In Definition~\ref{d:ALG-model} below, we will define the standard ALG model space
$(\cC_{\beta,\tau,L}(R), g^{\cC},\bm{\omega}^{\cC})$
 for parameters $L, R \in \RR_+$, and $(\beta, \tau)$ as in Table~\ref{ALGtable}. Here we just note that $\cC_{\beta,\tau,L}(R)$ is diffeomorphic to $(R, \infty) \times N_{\beta}^3$, where $N_{\beta}^3$ is a torus bundle over a circle, and the metric $g^{\cC}$ as well as the induced metric on the 3-manifold $N_{\beta}^3$ are flat;  the explicit formulae are given in Subsection~\ref{ss:ALGmodel}. We let $r$ denote the coordinate on $(R, \infty)$.

\begin{definition}[ALG gravitational instanton]
\label{d:ALG-space} 
A complete hyperk\"ahler $4$-manifold $(X,g, \bm{\omega})$
is called an ALG gravitational instanton of order $\mathfrak{n} > 0$   
with parameters  $(\beta,\tau)$ as in Table~\ref{ALGtable}, and $L > 0$ if there exist $R > 0$, a compact subset $X_R \subset X$, and a diffeomorphism 
$\Phi:   \cC_{\beta,\tau,L}(R) \rightarrow X \setminus X_R$,
such that
 \begin{align}
 \big|\nabla_{g^{\cC}}^k(\Phi^*g-g^{\cC})\big|_{g^{\cC}}&=O(r^{-k-\mathfrak{n}}), \\
 \big|\nabla_{g^{\cC}}^k(\Phi^*\omega_i- \omega_i^{\cC})\big|_{g^{\cC}}&=O(r^{-k-\mathfrak{n}}), \ i = 1,2,3, 
  \end{align}
as $r \to \infty$, for any $k\in\dN_0$.
\label{ALG-definition}
  \end{definition}

\begin{remark}
\label{r:ALG}
It was proved in \cite[Theorem~A]{CCII} that there exist $\ALG$ coordinates so that the order $\mathfrak{n}$ is $2$ in the $\I_0^*, \II, \III, \IV$ cases ($\beta \leq \frac12$)
and $ \mathfrak{n}=2-\frac{1}{\beta}$ in the $\II^*, \III^*, \IV^*$ cases  ($\beta > \frac12$).
\end{remark}

It was shown in \cite[Theorem~1.10]{CV} that any two ALG gravitational instantons with the same $\beta$ are diffeomorphic. So without loss of generality we can view any $\ALG$ gravitational instanton as living on a fixed space $X_{\beta}$.
Chen-Chen proved that the naive version of the Torelli-type theorem fails when $\beta > 1/2$; see \cite{CCIII}. Furthermore, it was shown in \cite[Theorem~1.12]{CV} that when $\beta > 1/2$, each ALG gravitational instanton of order~$2$ corresponds to a two-parameter family of ALG gravitational instantons with the same periods $[\bm{\omega}]$, with exactly one element of this family being of order~$2$.  
This reduces the general case to proving a Torelli uniqueness theorem for $\ALG$ gravitational instantons of order $2$, which is our next theorem.
\begin{theorem}[ALG Torelli uniqueness]  
\label{t:ALG-uniqueness}
Let $(X_{\beta}, g, \bm{\omega})$ and $(X_{\beta}, g', \bm{\omega}')$ be two $\ALG$ gravitational instantons with the same $\tau$ and $L$, which are both $\ALG$ of order $2$ with respect to a fixed $\ALG$ coordinate system on $X_{\beta}$. If 
\begin{align}
\label{e:wdr2}
[\bm{\omega}] = [\bm{\omega'}] \in H^2_{\mathrm{dR}}(X_{\beta}) \times H^2_{\mathrm{dR}}(X_{\beta}) \times H^2_{\mathrm{dR}}(X_{\beta}),
\end{align}
then there is a diffeomorphism $\Psi: X_{\beta} \rightarrow X_{\beta}$ 
which induces the identity map on $H^2_{\mathrm{dR}}(X_{\beta})$
such that $\Psi^* g' = g$ and $\Psi^*\bm{\omega'} =  \bm{\omega}$.
\end{theorem}

This will be proved using a modification of the gluing construction in \cite{CVZ}
(see Theorem~\ref{t:ALG-gluing} below), and then invoking the Torelli Theorem for K3 surfaces. We remark that the order 2 condition is essential to control the error term in the gluing construction. Moreover, the assumption that both hyperk\"ahler structures are ALG of order $2$ in a {\textit{fixed}} coordinate system is also crucial for the proof. However, it is superfluous in the following sense:  it was proved in \cite[Theorem~1.11]{CV} that any two ALG gravitational instantons of order $2$ with the same $(\beta, \tau)$ and $L$ can be pulled-back to a \textit{fixed} space $X_{\beta}$ such that they are both ALG of order $2$ in a \textit{fixed} ALG coordinate system $\Phi_{X_{\beta}}$ (after possibly modifying one of the ALG coordinate systems). 
This motivates the following definition. 
\begin{definition}
\label{d:ALGperiod} 
Let $\cM_{\beta, \tau, L}$ to be the collection of all gravitational instantons on $X_{\beta}$ with parameters   $\beta, \tau$, and $L$
which are ALG of order $2$ with respect to a fixed ALG coordinate system 
$\Phi_{X_{\beta}}$. For $(X_{\beta}, g^0, \bm{\omega}^0) \in \mathcal{M}_{\beta, \tau, L}$, 
the period map based at $\bm{\omega}^0$, $\mathscr{P} : \mathcal{M}_{\beta, \tau, L} \rightarrow 
\mathscr{H}^2 \oplus\mathscr{H}^2  \oplus \mathscr{H}^2$, is defined by 
\begin{align}
\label{Pdef}
\mathscr{P}( \bm{\omega} ) =( [\omega_1 - \omega_1^0], [\omega_2 - \omega_2^0], [\omega_3 - \omega_3^0]),
\end{align}
where $\mathscr{H}^2 \equiv  \Ima (H^2_{\mathrm{cpt}}(X_{\beta})\to H^2(X_{\beta}))$.
\end{definition}
We will show that $\mathscr{P}$  is well-defined in Section~\ref{s:period}.
The following is our main result about the period mapping in the ALG cases.
\begin{theorem} 
\label{t:ALGperiod}
If $(X_{\beta}, g, \bm{\omega}) \in \mathcal{M}_{\beta, \tau, L}$, then
\begin{align}
\label{ALGnondegcond}
\bm{\omega}[C] 
\not=(0,0,0) \mbox{ for all } [C] \in H_2(X_{\beta}; \ZZ)
\mbox{ satisfying }   [C]^2=-2.
\end{align}
Furthermore,  the period map $\mathscr{P}$ is surjective onto cohomology triples in $\mathscr{H}^2 \oplus\mathscr{H}^2  \oplus \mathscr{H}^2$ satisfying \eqref{ALGnondegcond}.
\end{theorem}
We will prove this in Section~\ref{s:period}. 
In particular, we see that the space of order $2$ $\ALG$ gravitational instantons with fixed parameters $\beta$ ,$\tau$, and $L$ has dimension $3(b_2(X_{\beta}) - 1)$, where $b_2(X_{\beta})$ is given in Table~\ref{ALGtable}. 

\subsection{ALG$^*$ gravitational instantons}
\label{ss:ALGstar}

In Section~\ref{s:model-metric}, we will define the 
{\textit{standard $\ALG^*$ model space}}, which is denoted by 
\begin{align}
(\fM_{2\nu}(R),g^{\fM}_{\kappa_0,L},  \bm{\omega}_{\kappa_0,L}^{\fM}) \equiv (\fM_{2\nu}(R), L^2 g^{\fM}_{\kappa_0}, L^2 \bm{\omega}_{\kappa_0}^{\fM}),
\end{align}
which depends on  parameters $\nu \in \ZZ_+$, $\kappa_0 \in \RR$, $R>0$, and an overall scaling parameter $L > 0$.
Here, we just note that the manifold $\fM_{2\nu}(R)$ is diffeomorphic to $(R, \infty) \times \mathcal{I}_{\nu}^3$, 
where $\mathcal{I}_{\nu}^3$ is an infranilmanifold, which is a circle bundle of degree $\nu$ over a Klein bottle. We will let $r$ denote the coordinate on $(R, \infty)$, $V$ denote the function $\kappa_0 + \frac{\nu}{\pi} \log r$ and $\fs$ denote the function $r V^{1/2}$. The hyperk\"ahler structure is obtained via a Gibbons-Hawking anstaz. See Section~\ref{s:model-metric} for explicit formulae.

\begin{definition}[ALG$^*$ gravitational instanton]
\label{d:ALGstar} 
A complete hyperk\"ahler $4$-manifold $(X,g,\bm{\omega} )$
is called an $\ALG^*$ gravitational instanton of order $\mathfrak{n} > 0$ with parameters 
$\nu \in \ZZ_+$, $\kappa_0 \in \RR$ and $L > 0$ if there exist an $\ALG^*$ model space $(\fM_{2\nu}(R), g^{\fM}_{\kappa_0,L}, \bm{\omega}^{\fM}_{\kappa_0,L})$ with $R >0$, a compact subset $X_R \subset X$, and a diffeomorphism 
$\Phi: \fM_{2\nu}(R) \rightarrow X \setminus X_R$ such that
  \begin{align}
 \big|\nabla_{g^{\fM}}^k(\Phi^*g- g^{\fM}_{\kappa_0,L})\big|_{g^{\fM}}&=O(\fs^{-k-\mathfrak{n}}), \\
 \big|\nabla_{g^{\fM}}^k(\Phi^*\omega_i- \omega_{i,\kappa_0,L}^{\fM})\big|_{g^{\fM}}&=O(\fs^{-k-\mathfrak{n}}), \ i = 1,2, 3,
\end{align}
as $\fs \to \infty$ for any $k\in\dN_0$. 
\end{definition}

\begin{remark}
It was proved in  \cite[Theorem~1.9]{CVZ2} that there exist $\ALG^*$ coordinates on $X$ so that the order satisfies $\mathfrak{n} \geq 2$. This decay order will be crucial in the proof of Theorem~\ref{t:ALGstar2} below. 
\end{remark}
It was proved in \cite[Theorem~1.6]{CV} that any $2$ $\ALG^*$ gravitational instantons
with the same $\nu$, where $1 \leq \nu \leq 4$ are diffeomorphic to each other. 
So without loss of generality we can view any $\ALG^*$ gravitational instanton as living on a fixed space $X_{\nu}$. With this understood, our next theorem is a Torelli uniqueness theorem for ALG$^*$ gravitational instantons.
\begin{theorem}[ALG$^*$ Torelli uniqueness] 
\label{t:ALGstar2} Let $1 \leq \nu \leq 4$, and $(X_{\nu},g, \bm{\omega})$, $(X_{\nu}, g', \bm{\omega}')$ be two $\ALG^*$ gravitational instantons with the same parameters $\kappa_0$ and $L$, which are both $\ALG^*$ of order~$2$ with respect to a fixed $\ALG^*$ coordinate system on $X_{\nu}$. If 
\begin{equation}
\label{e:wdr}
[\bm{\omega}] = [\bm{\omega'}]\in H^2_{\mathrm{dR}}(X_{\nu}) \times H^2_{\mathrm{dR}}(X_{\nu}) \times H^2_{\mathrm{dR}}(X_{\nu}),
\end{equation}
then there is a diffeomorphism $\Psi: X_{\nu} \rightarrow X_{\nu}$ 
which induces the identity map on $H^2_{\mathrm{dR}}(X_{\nu})$
such that $\Psi^* g' = g$ and $\Psi^*\bm{\omega'} =  \bm{\omega}$.
\end{theorem}
This will be proved using a new gluing construction:  we obtain hyperk\"ahler metrics on the K3 surface using $\ALG^*$ gravitational instantons
(see Subsection \ref{ss:K3} below), and then we invoke the Torelli Theorem for K3 surfaces. In our proof, the requirement that both metrics are ALG$^*$ of order $2$ with respect to a \textit{fixed} ALG$^*$ coordinate system is crucial.  However, this assumption is actually superfluous in the following sense. It was proved in \cite[Theorem~1.7]{CV} that
if $(X, g, \bm{\omega})$ and $(X',g', \bm{\omega'})$ are any two ALG$^*$ gravitational instantons of order $2$ with the same parameters $\nu, \kappa_0$, and $L$, then after possibly changing the ALG$^*$ coordinate system $\Phi'$ on $X'$, we can arrange that the diffeomorphism map commutes with $\Phi$ and $\Phi'$. So we can actually view any ALG$^*$ gravitational instanton with parameters $\nu, \kappa_0$, and $L$  as a gravitational instanton of order $2$ on a \textit{fixed} space $X_{\nu}$ with a \textit{fixed} ALG$^*$ coordinate system $\Phi_{X_\nu}$. 
Similar to the ALG case, we make the following definition. 
\begin{definition} 
\label{d:ALGstarperiod}
Define $\mathcal{M}_{\nu, \kappa_0, L}$ to be the collection of all gravitational instantons on $X_{\nu}$ with parameters   $\nu, \kappa_0$, and $L$
which are $\ALG^*$ of order $2$ with respect to a fixed $\ALG^*$ coordinate system 
$\Phi_{X_{\nu}}$. For $(X_{\nu}, g^0, \bm{\omega}^0) \in \mathcal{M}_{\nu, \kappa_0, L}$, 
the period map based at $\bm{\omega}^0$, $\mathscr{P} : \mathcal{M}_{\nu, \kappa_0, L} \rightarrow \mathscr{H}^2 \oplus\mathscr{H}^2  \oplus \mathscr{H}^2$, is defined as in \eqref{Pdef},where $\mathscr{H}^2 \equiv  \Ima (H^2_{\mathrm{cpt}}(X_{\nu})\to H^2(X_{\nu}))$.
\end{definition}
The following is our main result about the period mapping in the ALG$^*$ cases. 
\begin{theorem}
\label{t:ALGstarperiod}
 If $(X_{\nu}, g, \bm{\omega}) \in \mathcal{M}_{\nu, \kappa_0, L}$, then
\begin{align}
\label{ALGstarnondegcond}
\bm{\omega}[C] 
\not=(0,0,0) \mbox{ for all } [C] \in H_2(X_{\nu}; \ZZ)
\mbox{ satisfying }   [C]^2=-2.
\end{align}
Furthermore,  the period mapping $\mathscr{P}$ is an open mapping into the space of cohomology triples in $\mathscr{H}^2 \oplus\mathscr{H}^2  \oplus \mathscr{H}^2$satisfying \eqref{ALGstarnondegcond}.
\end{theorem}
We will prove this in Section~\ref{s:period}. 
In particular, we see that the space of order $2$ $\ALG^*$ gravitational instantons with fixed parameters $\nu$ ,$\kappa_0$, and $L$ has dimension $3(b_2(X_{\nu}) - 1) = 12 - 3 \nu$. We conjecture that the period mapping $\mathscr{P}$ is also surjective in the $\ALG^*$ cases.

\subsection{ALG$^*$ bubbles from the K3 surface}
\label{ss:K3}

In \cite{GW}, hyperk\"ahler metrics were constructed on elliptic K3 surfaces with 24 $\I_1$ fibers, which have a $2$-dimensional Gromov-Hausdorff limit $(\PP^1, d_{ML})$, where $d_{ML}$ is called the McLean metric. This was generalized to arbitrary elliptic K3 surfaces in \cite{GTZ}; see also \cite{OdakaOshima}. Subsequently, the authors gave a new construction on arbitrary elliptic K3 surfaces in \cite{CVZ}, which also allowed for a detailed description of the degeneration near the singular fibers, which we briefly describe next. Away from singular fibers, the degeneration is modeled by Greene-Shapere-Vafa-Yau's semi-flat metric; see \cite{GSVY}. 
A generalization of the Ooguri-Vafa metric (see \cite{OV}), which we called a multi-Ooguri-Vafa metric (with $b$ monopole points) were used to describe the degeneration near singular fiber of type $\I_b$. 
ALG metrics were used to describe the degeneration near fibers with finite monodromy. 
In the case of $\I_\nu^*$ fibers, the model used was a $\ZZ_2$-quotient of certain multi-Ooguri-Vafa metrics with $2 \nu$ monopole points, together with $4$ Eguchi-Hanson metrics due to the $4$ orbifold singularities of the resulting quotient. It was moreover shown in \cite{CVZ} that such degenerations exist for metrics which are K\"ahler with respect to the fixed elliptic complex structure. 

In this paper, let $\cK$ be an elliptic $\K3$ surface with a singular fiber $D^*$
of type  $\I_b^*$ and $(18-b)$ singular fibers of type $\I_1$, where  $1 \leq b \leq 14$ (recall that an elliptic K3 surface can have up to an $\I_{14}^*$ fiber \cite{Shioda2003}). Let $X$ be an $\ALG^*$ gravitational instanton of order $2$ with parameters $1 \leq \nu \leq 4$, $\kappa_0\in\RR$ and $L>0$. Near $\I_1$ fibers, we use the Ooguri-Vafa metric as before. Near $D^*$, we cut out a neighborhood of $D^*$ in $\cK$ and, as a new method, glue it with a neck region and a rescaling of $X$. We call the glued manifold $\cM_\lambda$.
\begin{theorem}
\label{t:K3}
There exists a family of hyperk\"ahler metrics $g_\lambda$ on the K3 surface $\cM_\lambda$ such that 
$(\cM_\lambda, g_\lambda) \xrightarrow{GH} (\PP^1,d_{ML})$ as $\lambda\to0$,
and such that near $D^*$, the rescaling limits are $X$ together with $b + \nu$ Taub-NUT bubbles.
\end{theorem}
In this case of an $\I_b^*$ fiber, the construction in \cite{CVZ} was done to preserve the elliptic complex structure. In this new gluing construction, the original elliptic complex structure is not preserved. 
An interesting question is to describe more precisely the complex structure degeneration of this new family.
We also point out that this construction is somewhat analogous to \cite{HSVZ} in that we construct a neck region with nontrivial topology which interpolates between different degree infranilmanifolds (versus nilmanifolds in \cite{HSVZ}), and which is responsible for the Taub-NUT bubbles. The proof of Theorem~\ref{t:K3} is contained in Sections~\ref{s:gluing-construction}, \ref{s:regularity}, and \ref{s:perturbation}.

\subsection{Acknowledgements} The authors would like to thank Hans-Joachim Hein, Rafe Mazzeo, and Song Sun for valuable discussions about gravitational instantons.

\section{The model hyperk\"ahler structures}
\label{s:model-metric}
In this section, we explain some properties of $\ALG$ and $\ALG^*$ gravitational instantons in more detail.

\subsection{Gibbons-Hawking construction}
\renewcommand{\t}[1]{\theta_{#1}}
In this subsection, we review the Gibbons-Hawking construction of the ALG$^*$ model metric. See \cite{CVZ2} for more details. For any positive integer $\nu$, the Heisenberg nilmanifold $\Nil^3_{2\nu}$ of degree $2\nu$ is the quotient of $\RR^3$ by the following actions
\begin{align}
\label{group1}
\sigma_1 (\theta_1,\theta_2 , \theta_3) &\equiv (\theta_1 + 2 \pi, \theta_2, \theta_3 ),\\
\label{group2}
\sigma_2 (\theta_1,\theta_2 , \theta_3)&\equiv (\theta_1, \theta_2+ 2 \pi, \theta_3 + 2 \pi \theta_1),\\
\label{group3}
\sigma_3 (\theta_1,\theta_2 , \theta_3)&\equiv ( \theta_1, \theta_2, \theta_3 +  2 \pi^2 \nu^{-1} ).
\end{align}
Define
\begin{align}
\Theta \equiv \frac{\nu}{\pi} (d \theta_3 - \theta_2 d\theta_1), \quad V \equiv \kappa_0 + \frac{\nu}{\pi} \log r, \end{align}
 for $r\in(R,\infty)$, $\kappa_0\in\dR$, and $R>e^{\frac{\pi}{\nu}(1-\kappa_0)}$
on the manifold
\begin{align}
S^1 \to \widehat{\fM}_{2\nu}(R) \equiv (R, \infty) \times \Nil^3_{2\nu} \to \widetilde{U} \equiv (\RR^2 \setminus \overline{B_{R}(0)}) \times S^1.
\end{align}
Then the Gibbons-Hawking metric on $\widehat{\fM}_{2\nu}(R)$ is given by
\begin{align}
\begin{split}
g^{\widehat{\fM}}_{\kappa_0} & = V ( dr^2 + r^2 d\theta_1^2 + d \theta_2^2) + V^{-1} \frac{\nu^2}{\pi^2} \Big( d \theta_3 - \theta_2 d \theta_1 \Big)^2 \\ 
& = V ( dx^2 + dy^2 + d\theta_2^2) + V^{-1} \Theta^2,
\end{split}\label{metricexp}
\end{align}
where $x + \i y \equiv r\cdot e^{\i \t1}$.
The model hyperk\"ahler forms on the manifold $\widehat{\fM}_{2\nu}(R)$ are given by
\begin{align}
\label{mkf1}
\omega_I&=\omega_{1, \kappa_0}^{\widehat{\fM}}  = E^1 \wedge E^2 + E^3 \wedge E^4= V dx \wedge dy + d \t2 \wedge \Theta,\\
\label{mkf2}
\omega_J&= \omega_{2, \kappa_0}^{\widehat{\fM}} = E^1 \wedge E^3 - E^2 \wedge E^4 = V dx \wedge d \t2 - dy \wedge \Theta,\\
\label{mkf3}
\omega_K&=\omega_{3, \kappa_0}^{\widehat{\fM}} = E^1 \wedge E^4 + E^2 \wedge E^3 = dx \wedge \Theta + V dy \wedge d \t2,
\end{align}
where
\begin{align}
\{ E^1, E^2, E^3, E^4 \} = \{ V^{1/2} dx, V^{1/2} dy , V^{1/2} d\t2, V^{-1/2} \Theta\}.
\end{align}
The $\ZZ_2$-action
$\label{nilaction}
\iota (r, \t1,\t2, \t3) \equiv (r, \t1 + \pi, - \t2, - \t3)$
 induces an automorphism of the hyperk\"ahler structure, and
we define the $\ALG^*_{\nu}$ model space as
\begin{align*}
(\fM_{2\nu}(R), g_{\kappa_0}^{\fM},\omega_{1,\kappa_0}^{\fM}, \omega_{2, \kappa_0}^{\fM}, \omega_{3,\kappa_0}^{\fM})
\equiv
(\widehat{\fM}_{2\nu}(R), g_{\kappa_0}^{\widehat{\fM}}, \omega_{1, \kappa_0}^{\widehat{\fM}}, \omega_{2, \kappa_0}^{\widehat{\fM}}, \omega_{3,\kappa_0}^{\widehat{\fM}})/\langle \iota \rangle.
\end{align*}
  By rescaling, we have $(\fM_{2\nu}(R), g_{\kappa_0, L}^{\fM}, \omega_{1, \kappa_0, L}^{\fM}, \omega_{2, \kappa_0, L}^{\fM}, \omega_{3,\kappa_0, L}^{\fM})$ for any scaling parameter $L>0$
, where 
  \begin{align*}g_{\kappa_0, L}^{\fM}\equiv L^2 \cdot g_{\kappa_0}^{\fM}, \quad \omega_{i, \kappa_0, L}^{\fM}\equiv  L^2 \cdot  \omega_{i, \kappa_0}^{\fM}, \ i = 1,2,3,\end{align*}

\begin{remark} The model space has the following properties. The cross-section $r = r_0$ is a \textit{infranil} $3$-manifold.  
There is a holomorphic map $u_{\fM}:\fM_{2\nu}(R)\to\mathbb{C}$ defined as $u_{\fM} = r^2 e^{2 \i \theta_1} $, with torus fibers.  The infinite end of the model space compactifies 
complex analytically by adding a singular fiber of type $\I_{\nu}^*$.
\end{remark} 

\subsection{Choice of connection form} 
In this subsection, we make some important remarks about our choice of connection form. 
The connection form satisfies 
\begin{align}
d \Theta =  \frac{\nu}{\pi}  d \theta_1 \wedge d \theta_2
\end{align}
and $\iota^* \Theta = - \Theta$.
Since $\dim ( H^1_{\mathrm{dR}}(\widetilde{U})) =2$ and is generated by $d\theta_1$ and $d\theta_2$,  more generally we could have chosen 
\begin{align}
\widetilde{\Theta} =  \frac{\nu}{\pi} \Big( d \theta_3 - \theta_2 d\theta_1
+ d f + p d \theta_1 + q d \theta_2 \Big),
\end{align}
where $f : \widetilde{U} \rightarrow \RR$, and $p, q \in \RR$. Note that  $\iota^* \widetilde{\Theta} = - \widetilde{\Theta}$ implies that $p=0$ and 
\begin{align}
f(r, \t1, \t2) + f(r, \t1+\pi, - \t2) = c
\end{align}
for a constant $c\in\RR$.
The mapping 
\begin{align}
\varphi_f (r, \theta_1, \theta_2, \theta_3) \equiv  (r, \theta_1, \theta_2, \theta_3 +\frac{c}{2} - f)
\end{align}
commutes with $
\sigma_1$, $\sigma_2$, $\sigma_3$ and $\iota$. Moreover, we have
\begin{align}
\varphi_f^* \widetilde{\Theta} = \frac{\nu}{\pi} \Big( d \theta_3 - \theta_2 d\theta_1 + q d \theta_2 \Big).
\end{align}
Next, define the mapping 
\begin{align}
\varphi_{q} (\theta_1, \theta_2, \theta_3) \equiv (\theta_1 - q , \theta_2, \theta_3  - q \theta_2).
\end{align}
It is straightforward to compute that $\varphi_{q}$ also commutes with $
\sigma_1$, $\sigma_2$, $\sigma_3$ and $\iota$.
Clearly, we have 
$\varphi_{q}^* \varphi_f^* \widetilde{\Theta} = \Theta$.
 Remark that the mapping $\varphi_f \circ \varphi_{q}$ is clearly an isometry of the Gibbons-Hawking metric $g^{\widehat{\fM}}_{\kappa_0}$.
Since the mapping $\varphi_f \circ \varphi_{q}$ induces a diffeomorphism
 $\varphi_f \circ \varphi_{q} :\widehat{\fM}_{2\nu}(R)/\iota \rightarrow  \widehat{\fM}_{2\nu}(R)/\iota$,
this mapping is an isometry of the quotient metric. Therefore, we may assume without loss of generality that $f = 0$ and $p = q =0$, so any choice of connection form is equivalent to $\Theta$ up to diffeomorphism.

\begin{remark}
If we replace $\Theta$ in \eqref{mkf1}, \eqref{mkf2}, \eqref{mkf3} by
\begin{align}
\widetilde{\Theta} =  \frac{\nu}{\pi} \Big( d \theta_3 - \theta_2 d\theta_1
+ d f + q d \theta_2 \Big),
\end{align}
to get $\widetilde{\omega}_I$, $\widetilde{\omega}_J$, $\widetilde{\omega}_K$, then 
\begin{align*}
\varphi_{q}^* \varphi_f^* (\widetilde{\omega}_I, \widetilde{\omega}_J, \widetilde{\omega}_K) = (\omega_I, \cos q \cdot \omega_J + \sin q \cdot \omega_K, \cos q \cdot \omega_K - \sin q \cdot \omega_J).
\end{align*}
In other words, we can use the standard $\Theta$ after a hyperk\"ahler rotation.
\end{remark}

\subsection{ALG model space}
\label{ss:ALGmodel}
In the ALG case, we have the following definition of the model space. 
 \begin{definition}[Standard ALG model]\label{d:ALG-model}
  Let $\beta\in(0,1]$, and $\tau\in\mathbb{H}\equiv\{\tau\in\dC|\Ima\tau>0\}$ be parameters in Table~\ref{ALGtable}, and  $L > 0$ be a scaling parameter.
Consider the space 
\begin{equation}\{(\mathscr{U},\mathscr{V}) \ | \ \arg \mathscr{U}\in[0,2\pi\beta]\}\subset(\mathbb{C}\times \mathbb{C})/(\dZ\oplus \dZ),
\end{equation}
where $\dZ\oplus \dZ$ acts on $\mathbb{C}\times \mathbb{C}$ by
\begin{equation}
(m,n)\cdot (\mathscr{U},\mathscr{V})= \Big(\mathscr{U}, \mathscr{V}
 + (m+n\tau) \cdot L \Big),
\ (m,n)\in\dZ\oplus \dZ.
\end{equation}
We can furthur identify $(\mathscr{U},\mathscr{V})$ with $(e^{\i\cdot 2\pi  \beta}\mathscr{U},e^{- \i  \cdot 2\pi\beta}\mathscr{V})$ to obtain a manifold $\cC_{\beta,\tau,L}$.
Define 
\begin{equation}
\cC_{\beta,\tau,L}(R) \equiv \{|\mathscr{U}|>R\} \subset \cC_{\beta,\tau,L}.
\end{equation}
Then there is a flat hyperk\"ahler metric 
\begin{equation}
   g^{\cC}=\frac{1}{2}(d\mathscr{U}\otimes d\mathscr{\bar U} + d\mathscr{\bar U} \otimes d \mathscr{U}+d\mathscr{V}\otimes d\mathscr{\bar V} + d\mathscr{\bar V} \otimes d \mathscr{V})
  \end{equation}
 on $\cC_{\beta,\tau,L}(R)$ with a hyperk\"ahler triple
\begin{align*}
   \omega_1^{\cC} &= \frac{\i}{2}(d\mathscr{U}\wedge d\mathscr{\bar U}+d \mathscr{V}\wedge d\mathscr{\bar V}), \\
   \omega_2^{\cC} &= \Rea (d\mathscr{U}\wedge d \mathscr{V}), \quad
   \omega_3^{\cC} = \Ima (d\mathscr{U}\wedge d \mathscr{V}).
 \end{align*}
Each flat space $(\cC_{\beta,\tau,L}(R), g^{\cC},\bm{\omega}^{\cC})$ given as the above  is called a {\it standard $\ALG$ model}.
\end{definition}

\begin{table}[h]
\caption{Invariants of ALG spaces.}
\label{ALGtable}
 \renewcommand\arraystretch{1.5}
\begin{tabular}{|c|c|c|c|c|c|c|c|c|} \hline
$\infty$ & $\I_0^*$ & $\II$ & $\II^*$ & $\III$ & $\III^*$ & $\IV$ & $\IV^*$\\\hline
$\beta\in(0,1]$ &  $\frac{1}{2}$  & $\frac{1}{6}$ & $\frac{5}{6}$ & $\frac{1}{4}$ & $\frac{3}{4}$ & $\frac{1}{3}$ & $\frac{2}{3}$\\ [5pt]\hline
  $\tau\in\mathbb{H}$ & Any & $e^{\i \cdot \frac{2\pi}{3}}$ & $e^{\i \cdot \frac{2\pi}{3}}$ & $\i$ & $\i$ & $e^{\i \cdot \frac{2\pi}{3}}$ & $e^{\i \cdot \frac{2\pi}{3}}$ \\\hline
$b_2(X_{\beta})$  & 5 & 9 & 1 & 8 & 2 & 7 & 3 \\ \hline
\end{tabular}
\end{table}

\begin{remark} The model space has the following properties. Letting $r = |\mathscr{U}|$,  the cross-section $\{ r = r_0 \}$ is a \textit{flat} $3$-manifold.  
There is a holomorphic map $u_{\cC}:\cC_{\beta,\tau,L}(R)\to\mathbb{C}$ defined as $u_{\cC} = \mathscr{U}^{\frac{1}{\beta}}$, with torus fibers, which have area $L^2 \cdot \Ima \tau$.  The infinite end of the model space compactifies 
complex analytically by adding a singular fiber of the specified type in the first row of Table \ref{ALGtable}.
\end{remark}

\section{Building blocks and approximate metrics}
\label{s:gluing-construction}
In this section, we will describe the construction of the ``approximate'' hyperk\"ahler triple, using a gluing construction. We will divide the K3 surface into the following regions: the $\ALG_{\nu}^*$ bubbling region, the Gibbons-Hawking neck transition region, the Ooguri-Vafa regions, and the collapsing semi-flat hyperk\"ahler structure away from singular fibers.

We start with an elliptic K3 surface $\pi_{\mathcal{K}}: \cK\to \PP^1$ with an $\I_b^*$ fiber for some $1 \leq b \leq 14$ and $\I_1$ fibers of number $(18-b)$. 
Away from all singular fibers, we choose the hyperk\"ahler structure as $\bm{\omega}^{\SF}$, given by the Greene-Shapere-Vafa-Yau ansatz; see \cite[Subsection 2.2]{CVZ}. 
Near the $\I_1$ fibers, we glue in Ooguri-Vafa metrics as in \cite{GW, CVZ}. These regions contribute exponentially small error terms to the weighted estimates, so in the following we will take this as understood, and will not consider those regions in any detail.  We will denote this region of the K3 surface by $\cK^* = \cK \setminus D^*$, where $D^*$ is the $\I_b^*$ fiber, and will continue to denote the hyperk\"ahler triple on this region by $\bm{\omega}^{\SF}$, even though it is not semi-flat near the $\I_1$ fibers.

Near the $\I_b^*$ fiber, as in \cite{CVZ}, we consider the local double cover, which is an $\I_{2b}$ fiber. We choose local coordinate $\mathscr{Y}$ on the base of the local double cover and local coordinate 
$\mathscr{X} \in \CC / (\ZZ \tau_1(\mathscr{Y})\oplus \ZZ \tau_2(\mathscr{Y}))$
on the fiber of the local double cover such that $\Omega = d \mathscr{X}  \wedge d\mathscr{Y}$,
and for some holomorphic function $h(\mathscr{Y})$,
\begin{align} \label{e:tau2} 
\tau_1(\mathscr{Y}) = 1\quad 
 \text{and}\quad 
\tau_2(\mathscr{Y}) = \frac{b}{\pi \i} \log \mathscr{Y} + h (\mathscr{Y})
\end{align}

\subsection{$\ALG_{\nu}^*$ bubbling region}

 Given a fixed $\nu\in\{1,2,3,4\}$, let $(X, g^{X}, \bm{\omega}^X)$ be an $\ALG_{\nu}^*$ gravitational instanton with parameters $\nu, \kappa_0$, and $L$. Without loss of generality, by scaling we can assume that $L = 1$. Recall the model space is the $\dZ_2$-quotient of the Gibbons-Hawking model $\fM_{2\nu}(R)$, where
the Riemannian metric $g^{\widehat{\fM}}$ and hyperk\"ahler triple $\bm{\omega}^{\widehat{\fM}}$ of the $\dZ_2$-covering space $\widehat{\fM}_{2\nu}(R)$ are  given by the following explicit formulae (as in Section \ref{s:model-metric}) when $r$ is sufficiently large:
\begin{align}
\begin{split}
 g^{\widehat{\fM}} & = V ( dr^2 + r^2 d\theta_1^2 + d \theta_2^2) + V^{-1}  \Theta^2, \\
 \omega_1^{\widehat{\fM}} & = V dx\wedge dy + d\theta_2 \wedge \Theta, \\
 \omega_2^{\widehat{\fM}} & = V dx \wedge d \t2 - dy \wedge \Theta, \quad
 \omega_3^{\widehat{\fM}}  = dx \wedge \Theta + V dy \wedge d \t2.
 \end{split}
\label{e:ALG-model-function}
\end{align}
where $V =  \frac{\nu}{\pi}  \log r + \kappa_0$ and $\kappa_0\in\dR$.
To perform the gluing construction, we will take a large region in $X$ and appropriately scale down both $(g^{X},  \bm{\omega}^X)$ and $(g^{\fM}, \bm{\omega}^{\fM})$.
We will fix parameters $\lambda$ and $\ft$ such that
\begin{align}
\lambda\to 0, \quad \ft\to 0, \quad  \sigma \equiv \frac{\lambda}{\ft} \to 0.
\label{e:parameters-fixing}
\end{align}	
Let us consider the rescaled coordinates $\tilde{x} \equiv \lambda \cdot x$,  
 $\tilde{y} \equiv \lambda \cdot y$  
 for $(x,y)\in B_{\sigma^{-1}}(0^2)\subset \dR^2$.  Immediately, $\tilde{r}=(\tilde{x}^2 + \tilde{y}^2)^{\frac{1}{2}} = \lambda \cdot r$.
We will work with the cutoff region $X\setminus \{r>2\sigma^{-1}\}$ with the rescaled $\ALG_{\nu}^*$ hyperk\"ahler structure
$(\tilde{g}^{X}, \tilde{\bm{\omega}}^{X}) = (\lambda^2\cdot g^{X}, \lambda^2\cdot  \tilde{\omega}^{X})$.	
Then the rescaled metric and hyperk\"ahler triple on the asymptotic model can be written in terms of the rescaled coordinates:
\begin{align}
\widetilde{V}  & = T + \frac{\nu}{\pi} \cdot \log \tilde{r} + \kappa_0, \quad T \equiv \frac{\nu}{\pi}\log\left(\frac{1}{\lambda}\right) \gg 1, \\
\lambda^2 \cdot g^{\widehat{\fM}} & = \widetilde{V} (d \tilde{x}^2 + d \tilde{y}^2 + \lambda^2 d \theta_2^2) + \lambda^2 \cdot \widetilde{V}^{-1}\cdot \Theta^2, \\
\lambda^2 \cdot \omega_1^{\widehat{\fM}} & = \widetilde{V}  \cdot d \tilde{x} \wedge d \tilde{y} + \lambda^2 \cdot d \theta_2 \wedge \Theta, \\
\lambda^2 \cdot \omega_2^{\widehat{\fM}} & = \lambda \cdot \widetilde{V}  \cdot d \tilde{x} \wedge d \theta_2 - \lambda \cdot d \tilde{y} \wedge \Theta, \\
\lambda^2 \cdot \omega_3^{\widehat{\fM}} & = \lambda \cdot d \tilde{x} \wedge \Theta + \lambda \cdot \widetilde{V} \cdot d \tilde{y} \wedge d \theta_2.
	\label{e:ALG-harmonic-function}
 \end{align}
Note that the cutoff region becomes $\mathcal{O}_{2\ft}(p)\equiv X\setminus\{\tilde{r}>2\ft\}$ in terms of $\tilde{r}$.

\subsection{Neck transition region}
\label{ss:neck-transition-region}

The next building block is the neck transition region.
To begin with, we take a flat product metric on $Q\equiv \dR^2\times S^1=\dR^2\times (\dR/2\pi\dZ)$ with $0^*\equiv (0^2, 0)\in \dR^2\times S^1$,
\begin{align}
g^{Q} = d x^2 + d y^2 +   d \theta_2^2 = d r^2 + r^2 d \theta_1^2 +  d \theta_2^2,	\quad \theta_2\in[0,2\pi].
\end{align}
For fixed $\kappa_0$ in \eqref{e:ALG-model-function}  and small parameter $\lambda\ll 1$.
 Let
 \begin{align}\widetilde{P}=\{\tilde{p}_1,\tilde{p}_2, \ldots, \tilde{p}_{2 \nu + 2 b}\}\subset (\dR^2\setminus\{0^2\})\times \{0\}\subset Q\end{align} be a fixed set such that the following properties hold.
 \begin{enumerate}
 \item (Balancing condition) Let $\tilde{d}_m \equiv d^{Q}(0^*, \tilde{p}_m)$ for any $1 \leq m \leq 2 \nu + 2 b$. Then \begin{align}\sum\limits_{m=1}^{2 \nu + 2b}\log (1/\tilde{d}_m) + 2 \pi \Ima h(0) = 2 \pi \kappa_0,\label{e:balancing-condition}\end{align} where $h$ is the holomorphic function in \eqref{e:tau2}.
 \item ($\dZ_2$-invariance)	$\iota (\tilde{p}_m) = \tilde{p}_{2 \nu + 2 b + 1 - m}$ for any $1 \leq m \leq 2 \nu + 2 b$.
 \end{enumerate}
 Let $P$ be the dilation of the set $\widetilde{P}$ by $\lambda^{-1}$: we define
 \begin{align}
   p_m \equiv (\lambda^{-1}\cdot \tilde{x}_m, \lambda^{-1}\cdot \tilde{y}_m,0)\in(\dR^2\setminus\{0^2\})\times S^1,	\ 1 \leq m \leq 2 \nu + 2 b,
 \end{align}
where $\tilde{p}_m=(\tilde{x}_m,\tilde{y}_m,0)$.
 Then there are constants $\iota_0>0$ independent of $\lambda$ such that 
\begin{align}
  & \iota_0 \cdot \lambda^{-1}  \leq d^{Q}	(p_{\alpha},0^*)\leq \iota_0^{-1} \cdot \lambda^{-1},
\\
& \iota_0 \cdot \lambda^{-1} \leq d^{Q}(p_{\alpha},p_{\beta}) \leq \iota_0^{-1} \cdot \lambda^{-1}, \quad 1 \leq \alpha < \beta \leq 2 \nu + 2 b.
\end{align}
   For every $p_m\in P$ with $1\leq m\leq 2 \nu + 2 b$, there exists a unique Green's function $G_m$ on $(\dR^2\times S^1, g^Q)$ that satisfies $-\Delta_{g^Q} G_m = 2\pi \delta_{p_m}$ and has the asymptotics
\begin{align}
\label{e:Gm}
	\left|G_m - \frac{1}{2\pi}\log \frac{1}{d^Q(\ux, p_m)}\right| \leq C \cdot e^{- d^Q(\ux, p_m)} \quad \text{as}\  d^Q(\ux,P)\to\infty.
\end{align}
The proof is standard and we omit it. The above Green's function was also used in \cite{CVZ} to construct the neck transition region; see \cite[Lemma 4.1]{CVZ}. Let $G_0\equiv \sum\limits_{m=1}^{2 \nu + 2 b} G_m$ be the superposition that solves the equation $
-\Delta_{g^Q} G_0 = 2 \pi \sum\limits_{m=1}^{2 \nu + 2 b}\delta_{p_m}$.
 We also take  \begin{align}G_{\lambda} \equiv  \frac{2 \nu + b}{\pi}\log\left(\frac{1}{\lambda}\right)   + \frac{\nu}{\pi}\cdot\log\tilde{r} + G_0 + \Ima h\left(\tilde{\lambda} \cdot (\tilde x + \i \tilde y)
\right),\label{e:neck-harmonic-function}
\end{align}
where  $\tilde{\lambda}\equiv \lambda^{\frac{\nu}{b}}, r \equiv \lambda^{-1}\cdot \tilde{r},$
and $h$ is the holomorphic function defined in \eqref{e:tau2}.
Letting $T\equiv \frac{\nu}{\pi} \log\left(\frac{1}{\lambda}\right)$, we have $\tilde{\lambda}\equiv e^{-\frac{\pi T}{b}}$. 
Switching to the rescaled metric  $\tilde{g}^{Q} = \lambda^2 g^{Q}$,  let us discuss the asymptotic behavior of the Green's function $G_{\lambda}$ in terms of the distance function $\tilde{r}$, which will be used in the discussions of the rescaling geometry in the later subsections. 
Applying \eqref{e:balancing-condition} and \eqref{e:neck-harmonic-function},
then the following holds for any sufficiently small $\lambda\ll1$ and $\ux\in Q$.
\begin{enumerate}
\item (Near the origin) If $\tilde{r}(\ux) \to 0$, then $G_{\lambda}$ has the expansion
\begin{align}
G_{\lambda}(\ux) = 	T + \kappa_0 +  \frac{\nu}{\pi}\cdot\log\tilde{r}(\ux) + E(\ux),\label{e:expansion-around-the-origin}
\end{align}
where $|E(\ux)|\leq C \cdot \tilde{\lambda} \cdot \tilde{r}(\ux) = C \cdot \lambda^{\frac{\nu}{b}} \cdot \tilde{r}(\ux)$ for some constants $C>0$ independent of $\lambda$.

\item (Near the infinity of $Q$) If $\tilde{r}(\ux) \to \infty$, then
 \begin{align}
 \label{e:Greens-function-near-infinity}
 \begin{split}
&\ \left|G_{\lambda}(\ux) - \Ima \tau_2(\tilde{\lambda} \cdot \tilde{\zeta})\right|   \\
= &\  \Big|G_{\lambda}(\ux) - \Big(T- \frac{b}{\pi} \cdot \log \tilde{r}(\ux) + \Ima h(\tilde{\lambda} \cdot \tilde{\zeta} )\Big)\Big|   
  \leq   \frac{C\cdot \lambda^2}{\tilde{r}(\ux)^2},	\end{split}\end{align}
 where $\tilde{\zeta}\equiv \tilde{x} + \sq \tilde{y}$, $C>0$ is independent of $\lambda$, and $\tau_2$ is the function in \eqref{e:tau2}.

\item (Near a pole $p_m\in P$) If $\tilde{d}^Q(\ux, p_m) \leq \frac{\iota_0}{4}$ for some $p_m \in P$, then there exists a constant $C > 0$ independent of $\lambda$ such that
\begin{align}
\Big|G_{\lambda}(\ux) - \Big(G_m(\ux) +  T^{\flat} +   \frac{\nu}{\pi}\log\tilde{d}_m \Big)\Big|	\leq C,\quad T^{\flat}\equiv  \frac{2 \nu + 1}{2\pi}\cdot \log (\frac{1}{\lambda}).\label{e:Greens-function-asymptotics-pole}
\end{align}

\item (Bounded region) If there exist  $R_0>0$ and $d_0>0$  such that $R_0^{-1}\leq \tilde{r}(\ux)\leq R_0$ and $\tilde{d}^Q(\ux, P)    \geq \frac{d_0}{4}$, then
$|G_{\lambda}(\ux) - T| \leq C$,	
 where $C=C(R_0, d_0)>0$ is independent of $\lambda$.

\end{enumerate}

Now we apply the Gibbons-Hawking construction using the Green's function $G_{\lambda}$. Let $\mathring{\mathcal{N}}$ be the total space of the circle bundle
$S^1\to \mathring{\mathcal{N}}\xrightarrow{\pi} Q\setminus (P\cup \{0^*\})$ with the $S^1$-connection form $\Theta_{\lambda}$ that satisfies the monopole equation $d\Theta_{\lambda}= *_{g^Q} \circ dG_{\lambda}$.
Then we have a family of hyperk\"ahler metrics $g^{\mathring{\mathcal{N}}}$ and hyperk\"ahler triples $\bm{\omega}^{\mathring{\mathcal{N}}}$ when $G_{\lambda}>0$:
 \begin{align}\label{e:neck-metric}
 \begin{split}g^{\mathring{\mathcal{N}}} &= \lambda^2  (G_{\lambda}\cdot g^{Q} + G_{\lambda}^{-1}\Theta_{\lambda}^2) = G_{\lambda} (d \tilde{x}^2 + d \tilde{y}^2 + \lambda^2 d \theta_2^2) +\lambda^2  G_{\lambda}^{-1} \Theta_{\lambda}^2,	
 \\
 \omega_{1}^{\mathring{\mathcal{N}}}& = \lambda^2  (G_{\lambda} d x \wedge d y + d \theta_2 \wedge \Theta_{\lambda}) = G_{\lambda}  d \tilde{x} \wedge d \tilde{y} + \lambda^2  d \theta_2 \wedge \Theta_{\lambda}, \\
 \omega_{2}^{\mathring{\mathcal{N}}}& = \lambda^2  (G_{\lambda} d x \wedge d \theta_2 - d y \wedge \Theta_{\lambda}) = \lambda  G_{\lambda} \cdot d \tilde{x} \wedge d \theta_2 - \lambda  d \tilde{y} \wedge \Theta_{\lambda}, \\
 \omega_{3}^{\mathring{\mathcal{N}}}& = \lambda^2  (d x \wedge \Theta_{\lambda} + G_{\lambda}  d y \wedge d \theta_2) = \lambda  d \tilde{x} \wedge \Theta_{\lambda} + \lambda  G_{\lambda}  d \tilde{y} \wedge d \theta_2.
\end{split}\end{align}
It is easy to check that the completion $(\mathcal{N}, g^{\mathcal{N}}, \bm{\omega}^{\mathcal{N}})$ of $(\mathring{\mathcal{N}}, g^{\mathring{\mathcal{N}}},\bm{\omega}^{\mathring{\mathcal{N}}})$ along the set $P$ of monopole points, called {\it the neck transition region}, is smooth and hyperk\"ahler. Moreover, the neck transition region   $(\mathcal{N}, g^{\mathcal{N}}, \bm{\omega}^{\mathcal{N}})$  is invariant under the involution $\langle\iota\rangle\cong\dZ_2$, and hence it descends to a hyperk\"ahler manifold $(\mathfrak{N}, g^{\mathfrak{N}},\bm{\omega}^{\mathfrak{N}})$, where $\mathfrak{N} \equiv \mathcal{N}/\langle\iota\rangle$.

\subsection{Attaching the pieces}
\label{ss:attach}
Let $(X, g^{X}, \bm{\omega}^X)$ be an $\ALG_{\nu}^*$ gravitational instanton of order $2$.  We will next glue the end of $X$  onto the neck transition region $\mathcal{N}$ near the origin. 
By definition, there exists a compact subset  $X_R \subset X$, and a diffeomorphism 
$\Psi:\mathfrak{M} \to X \setminus X_R$ such that for any $k \in \dN$,
\begin{align}
|\nabla^k_{g^\fM}(\Psi^*\bm{\omega}^{X} - \bm{\omega}^{\fM})| \leq C_k \cdot (r\cdot V(r)^{\frac{1}{2}})^{- 2 - k}.\label{e:sharp-order}	
\end{align}
Thanks to the following lemma, we are able to compare the two hyperk\"ahler triples $\lambda^2 \cdot \bm{\omega}^{\widehat{\fM}}$ and 
$\bm{\omega}^{\mathcal{N}}$ as $\tilde{r}\to 0$. 

\begin{lemma}\label{l:neck-model-gauge-transformation}
There exists a diffeomorphism 
\begin{align}\Psi^{\mathcal{N}}:\{\bx\in\mathcal{N}| \ft\leq \tilde{r}(\bx)\leq 2\ft\}\longrightarrow\{\bx\in\widehat{\fM}|\ft\leq \tilde{r}(\bx)\leq 2\ft\}
\end{align}
such that $(\Psi^{\mathcal{N}})^* dr = dr$, 
$(\Psi^{\mathcal{N}})^* d\theta_1 = d\theta_1$, $(\Psi^{\mathcal{N}})^* d\theta_2 = d\theta_2$, 
and \begin{align}\quad (\Psi^{\mathcal{N}})^*\Theta = \Theta_{\lambda} +  \pi^* \zeta\end{align}
for some $1$-form $\zeta$ on 
$\{\ux\in Q| \ft\leq \tilde{r}(\ux)\leq 2\ft\}$ that satisfies $\iota^* \pi^*\zeta= - \pi^*\zeta$, and
\begin{align}
|\nabla_{g^{\mathcal{N}}}^k(\pi^*\zeta)| & \leq C_k\cdot \tilde{\lambda} \cdot \ft^{1-k} \cdot V(\sigma^{-1})^{-\frac{1+k}{2}},
\\
|\nabla_{g^{\mathcal{N}}}^k((\Psi^{\mathcal{N}})^*(\lambda^2 \cdot \bm{\omega}^{\widehat{\fM}})-\bm{\omega}^{\mathcal{N}})|& \leq C_k \cdot \tilde{\lambda} \cdot \ft^{1-k} \cdot V(\sigma^{-1})^{- \frac{2+k}{2} },
\end{align}
for any $k\in\dN_0$. 
  Moreover, $\Psi^{\mathcal{N}}$ descends to a diffeomorphism 
\begin{align}
	\Psi^{\mathfrak{N}}:\{\bx\in\mathfrak{N}| \ft\leq \tilde{r}(\bx)\leq 2\ft\}\longrightarrow\{\bx\in\fM|\ft\leq \tilde{r}(\bx)\leq 2\ft\}.
\end{align}
\end{lemma}

\begin{proof}
The proof is the same as that of \cite[Lemma 6.1]{HSVZ}. Here we only mention the major difference. 
First, both $\mathcal{N}$ and $\widehat{\fM}$ can be viewed as principal $S^1$-bundles over 
$\widetilde{U}\subset\dR^2\times S^1$ with the connections $\Theta_{\lambda}$ and $\Theta$  respectively, where $\widetilde{U}\equiv \dR^2 \setminus \overline{B_R(0^2)}$.
One can easily check that they have the same Euler number $2\nu$ when $\ft\leq \tilde{r}(\bx)\leq 2\ft $. 

Therefore, there exists a bundle isomorphism $F:  \mathcal{N} \to \widehat{\fM}$ 
which covers the identity map on $\widetilde{U}\times S^1$. Moreover, the curvature difference is given by
\begin{align}
F^*(d\Theta) - d\Theta_{\lambda}  = *_Q \circ d(E),
\end{align}
where $E\in C^{\infty}(Q)$ is the function given by the expansion \eqref{e:expansion-around-the-origin}.
Applying the asymptotic estimate in \eqref{e:expansion-around-the-origin}, we have that
\begin{align}
|*_Q\circ d(E)|_{g^{\mathcal{N}}} \leq C\cdot \tilde{\lambda} \cdot V(\sigma^{-1})^{-1}.	
\end{align}
Standard Hodge theory implies that there exist
a diffeomorphism $\Psi^{\mathcal{N}}$, a flat connection $\Theta_{\FF}$, and a $1$-form $\zeta$ on $\widetilde{U}\times S^1$ such that
\begin{align}
& (\Psi^{\mathcal{N}})^*\Theta - \Theta_{\lambda} = \Theta_{\FF} + \pi^*\zeta, \\
&	|\nabla_{g^{\mathcal{N}}}^k \zeta|  \leq C_k\cdot \tilde{\lambda} \cdot \ft^{1-k} \cdot V(\sigma^{-1})^{-\frac{1+k}{2}}.
\end{align}
As discussed in Section 2.2, the flat connection $\Theta_{\FF}$ can be removed by appropriately choosing a bundle diffeomorphism.
So the proof is done.
\end{proof}

\begin{lemma}\label{l:neck-triple-close-to-model-triple}
	There exists a triple of $1$-forms $\bm{\xi}$ on $\{\bx\in\mathcal{N}| \ft\leq \tilde{r}(\bx)\leq 2\ft\}$ such that $\iota^*\bm{\xi} = \bm{\xi}$ and
	\begin{align}
(\Psi^{\mathcal{N}})^*(\lambda^2 \cdot\bm{\omega}^{\widehat{\fM}}) - \bm{\omega}^{\mathcal{N}}  = d\bm{\xi}.
	\end{align}
Moreover, $\bm{\xi}$ satisfies the estimate \begin{align}
	|\nabla_{g^{\mathcal{N}}}^k\bm{\xi}| \leq C_k \cdot   \tilde{\lambda} \cdot \ft^3 \cdot  (\ft \cdot V(\sigma^{-1})^{\frac{1}{2}})^{-1-k} 
\end{align}
for any $k\in\dN_0$. Moreover, $\Psi^{\mathcal{N}}$ descends to a diffeomorphism 
\begin{align}
	\Psi^{\mathfrak{N}}:\{\bx\in\mathfrak{N}| \ft\leq \tilde{r}(\bx)\leq 2\ft\}\longrightarrow\{\bx\in\fM|\ft\leq \tilde{r}(\bx)\leq 2\ft\}.
\end{align}
\end{lemma}

Once we have Lemma \ref{l:neck-model-gauge-transformation}, the proof of Lemma \ref{l:neck-triple-close-to-model-triple} follows from the same arguments as in \cite[Proposition 6.2]{HSVZ}. We omit the details.
 
Next, we glue the cutoff region $\{r \leq 2\sigma^{-1}\} \subset X$ as introduced above into the neck region $\fN$. We define the diffeomorphism 
 \begin{align}
\Phi \equiv (\Psi \circ \Psi^{\fN})^{-1}
\end{align}
from $\{\sigma^{-1} \leq r \leq 2\sigma^{-1}\}\subset X$ to a subset   $\{\ft\leq \tilde{r}\leq 2\ft\}\subset \fN$. 
Combining Lemma \ref{l:neck-triple-close-to-model-triple} and the asymptotic estimate of an $\ALG^*$ gravitational instanton, we have the following. 
\begin{lemma}\label{l:ALG^*-error} There exists a triple of $1$-forms $\bm{\eta}^{X}$ on $\{\bx\in \fN|\ft\leq \tilde{r}(\bx)\leq 2\ft\}$ such that  $(\Phi^{-1})^*(\lambda^2 \cdot \bm{\omega}^{X}) - \bm{\omega}^{\fN} = d \bm{\eta}^{X}$ 
and satisfies the estimate
\begin{align}
|\nabla_{g^{\fN}}^k \bm{\eta}^{X}|_{g^{\fN}} \leq C_k \cdot (\lambda^2 + \tilde{\lambda} \cdot \ft^3) \cdot  (\ft \cdot V(\sigma^{-1})^{\frac{1}{2}})^{-1-k}
\end{align}
for any $k\in\dN_0$.
\end{lemma}
Next, we will glue a subset of $\cK^*$ onto the end of the neck region
with $\tilde{r}$ large. As shown in \cite[Construction 2.6]{GW} and \cite[Proposition 2.3]{CV}, the hyperk\"ahler triple $\tilde\lambda^{-2} \cdot \bm{\omega}^{\SF}$ of the rescaled semi-flat metric on $\mathcal{K}^*$, up to a $\dZ_2$-covering, can be written in terms of the Gibbons-Hawking ansatz by applying the harmonic function 
\begin{align}
\V_{\SF}\equiv \Ima \tau_2(\tilde{\lambda} \cdot \tilde{\zeta}) = T-\frac{b}{\pi}\log \tilde{r} + \Ima h(\tilde{\lambda}\cdot \tilde{\zeta})
\end{align}
 which is the leading term of  \eqref{e:Greens-function-near-infinity}. 
Then we have the following lemma.

\begin{lemma}
\label{l:semiflat-GH}   For any sufficiently small parameter $\lambda\ll1$,  let $r_{\lambda}$ be a large number such that 
 $1\leq G_{\lambda}(\bx)\leq 100$
 as $r_{\lambda}\leq \tilde{r}(\bx)\leq 2r_{\lambda}$.
 There exist a triple of $1$-forms $\bm{\eta}^{\SF}$ on $\{\bx\in\fN|r_{\lambda}\leq \tilde{r}(\bx)\leq 2r_{\lambda}\}$ and a diffeomorphism $\Phi^{\flat}$ from  $\{\bx\in\fN|r_{\lambda}\leq \tilde{r}(\bx)\leq 2r_{\lambda}\}$ to a subset of $\mathcal{K}^*$ such that for all $k\in\dN_0$,
\begin{align}
& \bm{\omega}^{\fN} - (\Phi^{\flat})^* (\tilde\lambda^{-2} \cdot\bm{\omega}^{\SF}) = d\bm{\eta}^{\SF},
\\
& |\nabla_{g^{\fN}}^k \bm{\eta}^{\SF}|_{g^{\fN}} \leq C_k \cdot \lambda^2 \cdot \tilde{\lambda}^{1+k}.\label{e:triple-estimate-near-neck-infinity} 
\end{align}
\end{lemma}

Notice that \eqref{e:triple-estimate-near-neck-infinity} follows from \eqref{e:Greens-function-near-infinity}, and $r_\lambda$ is comparable to $\tilde{\lambda}^{-1}$.

\vspace{0.3cm}

With the above preparations, we are ready to define the closed glued manifold on which we will construct a family of collapsing hyperk\"ahler metrics with a given $\ALG^*$ gravitational instanton bubbling out. Now let us take the neck transition region $\fN$ equipped with the hyperk\"ahler triple $\bm{\omega}^{\fN}$ for any $\lambda \ll 1$, as constructed in Section \ref{ss:neck-transition-region}. In the region $\{\bx \in \fN|\ft \leq \tilde{r}(\bx)\leq 2\ft\}$, we glue $\fN$ with the finite part $\{r \le 2\sigma^{-1}\}$ of an $\ALG^*$ gravitational instanton $X$ using the diffeomorphism $\Phi$. In the region $\{\bx \in \fN|r_{\lambda} \leq \tilde{r}(\bx)\leq 2 r_{\lambda}\}$, we attach $\fN$ with $\mathcal{K}^*$ using the diffeomorphism $\Phi^{\flat}$ as in Lemma \ref{l:semiflat-GH}. 
Using the above gluing maps, we obtain a closed smooth $4$-manifold $\mathcal{M}_{\lambda}$. Now we construct a family of approximately hyperk\"ahler triples $\tilde{\bm{\omega}}_{\lambda}$ on $\mathcal{M}_{\lambda}$.

\begin{lemma}[Approximate hyperk\"ahler triple]\label{l:approximate-triple}
  For any sufficiently small parameter $\lambda \ll 1$,  let $r_{\lambda}$ be a large number such that $1\leq G_{\lambda}(\bx)\leq 100$ as $r_{\lambda}\leq \tilde{r}(\bx)\leq 2r_{\lambda}$.
  Then 
  there exist two triples of $1$-forms $\bm{\eta}^{X}$ and $\bm{\eta}^{\SF}$ such that the glued definite triple
\begin{align}
\tilde{\bm{\omega}}_{\lambda} \equiv
\begin{cases}
 \lambda^2 \cdot \bm{\omega}^{X} , & \tilde{r}\leq \ft,
\\
 \bm{\omega}^{\fN} + d \Big(\varphi\cdot \bm{\eta}^{X} - \psi\cdot \bm{\eta}^{\SF}\Big), & \ft\leq \tilde{r}\leq 2 r_{\lambda},
\\
\tilde{\lambda}^{-2} \cdot \bm{\omega}^{\SF}, & \tilde{r}\geq 2 r_{\lambda},\end{cases}\label{e:gluing-metric-ALG}
\end{align}
 satisfies the following estimates with respect to associated Riemannian metric $\tilde{g}_{\lambda}$ for any $k\in\dN_0$,
\begin{align}
\begin{split}
&\sup\limits_{\ft\leq \fr\leq 2\ft}|\nabla_{\tilde{g}_{\lambda}}^k\left(\tilde{\bm{\omega}}_{\lambda} - (\Phi^{-1})^*(\lambda^2 \cdot \bm{\omega}^{X})\right)|_{\tilde{g}_{\lambda}}  \\ 
 \leq &\ C_k\cdot (\lambda^2 + \tilde{\lambda} \cdot \ft^3) \cdot (\ft \cdot V(\sigma^{-1})^{\frac{1}{2}})^{- 2 - k},\end{split}\label{e:error-near-ALG*}
\end{align}
\begin{align}
\sup\limits_{r_{\lambda} \leq \fr \leq 2 r_{\lambda}}\left|\nabla_{\tilde{g}_{\lambda}}^k\left(\tilde{\bm{\omega}}_{\lambda} -  (\Phi^{\flat})^*(\tilde{\lambda}^{-2} \cdot \bm{\omega}^{\SF})\right)\right|_{\tilde{g}_{\lambda}}
 \leq C_k\cdot \lambda^2 \cdot \tilde{\lambda}^{2 + k}, \label{e:error-near-SF}
\end{align}
where $\tilde{\lambda}\equiv e^{-\frac{\pi T}{b}} = \lambda^{\frac{\nu}{b}}$, $\varphi$ and $\psi$ are smooth cut-off functions satisfying  \begin{align}
\varphi = \begin{cases}
1, & \tilde{r}\leq \ft,
\\
0, & \tilde{r}\geq 2\ft,
\end{cases} \qquad \text{and} \qquad
 \psi = \begin{cases}
1, & \tilde{r}\geq 2r_{\lambda},
\\
0, & \tilde{r}\leq  r_{\lambda},
\end{cases}
\end{align}
and $\bm{\omega}^{\SF}$ is the hyperk\"ahler triple of the semi-flat metric of Greene-Shapere-Vafa-Yau \cite{GSVY} with area of each fiber equal to $\tilde{\lambda}\cdot\lambda$ and diameter comparable to $1$.
\end{lemma}
\begin{proof}
The proof is straightforward. The error estimate in the region  $\{\ft\leq \tilde{r}\leq 2\ft\}$ is given by Lemma \ref{l:ALG^*-error}, and the error estimate in $\{r_{\lambda}\leq \tilde{r}\leq 2r_{\lambda}\}$ is due to Lemma \ref{l:semiflat-GH}.
\end{proof}

 It turns out that the manifold is indeed diffeomorphic to the K3 surface, but for now we do not need this fact, we only need the following calculation of the betti numbers. 
\begin{corollary}For $\lambda$ sufficiently small, the smooth $4$-manifold 
$\mathcal{M}_{\lambda}$ satisfies
\begin{align}
b^1(\mathcal{M}_\lambda) = 0, \quad b^2_+(\mathcal{M}_{\lambda}) = 3, \quad 
b^2_-(\mathcal{M}_{\lambda}) = 19, \quad \chi(\mathcal{M}_\lambda) = 24. 
\end{align}
\end{corollary}
\begin{proof}
This is proved using a Mayer-Vietoris argument and the estimates in Lemma~\ref{l:approximate-triple}, which show that $\Lambda^2_+(\mathcal{M}_{\lambda})$ is a trivial bundle if $\lambda$ is  small. We omit the details which are similar to \cite[Proposition~6.6]{HSVZ}.  
\end{proof}

\section{Metric geometry and regularity scales}
\label{s:regularity}

To begin with, we will list the notations. We will always fix a small of parameter $\lambda \ll 1$.
\begin{enumerate}
\item Let us denote   $g_{\lambda} \equiv \tilde{\lambda}^2 \cdot \tilde{g}_{\lambda}$.
Then it holds that there is some constant $C_0>0$ independent of $\lambda$ such that $C_0^{-1}\leq \diam_{g_{\lambda}}(\cM_{\lambda})\leq C_0$.

\item We define the smoothing function $\fr$ of the distance function $\tilde{r}$ by:
\begin{align*}
\fr(\bx)=
\begin{cases}
	\lambda\cdot R_0, & \tilde{r}(\ux)\leq \lambda \cdot R_0,
	\\
	\tilde{r}(\ux), &  2\lambda \cdot R_0 \leq \tilde{r}(\ux) \leq r_{\lambda},
	\\
2 r_{\lambda}, & \tilde{r}(\ux) \geq 2 r_{\lambda},
	\end{cases}
\end{align*}
where $R_0$ is the constant $R$ in Definition \ref{d:ALGstar}, and $r_\lambda$ is the constant in Lemma \ref{l:semiflat-GH}.

\item Given  $\nu \in \dZ_+$, let $ T^{\flat}\equiv \frac{2 \nu + 1}{2 \pi}\cdot\log(\frac{1}{\lambda})$ and let $\fd$ be the following: 	
\begin{align*}
\fd(\bx)
=\begin{cases}
 (T^{\flat})^{-\frac{1}{2}} , & d^Q(\ux, p_m)\leq   (T^{\flat})^{-1} \\
& \text{for some}\ 1\leq m \leq 2\nu+2b,
\\
&
\\
(T^{\flat})^{\frac{1}{2}}\cdot d^Q(\ux,p_m), &  2  (T^{\flat})^{-1}\leq d^Q(\ux, p_m) \leq 1
\\ & \text{for some}\ 1\leq m \leq 2\nu+2b,
\\
&
\\
(T^{\flat} + \frac{1}{2 \pi} \log \frac{1}{d^Q(\ux,p_m)})^{\frac{1}{2}}\cdot d^Q(\ux,p_m), &  2 \leq d^Q(\ux, p_m) \leq \frac{\iota_0}{4} \cdot \lambda^{-1}
\\ &   \text{for some}\ 1\leq m\leq 2 \nu + 2 b.
\end{cases}
\end{align*}
 
\item
Let us define $T\equiv \frac{\nu}{\pi}\log(\frac{1}{\lambda})$ and a smooth function $\mathfrak{L}_T$ that satisfies
\begin{align*}
\mathfrak{L}_T(\bx)
\equiv
\begin{cases}
1, & \tilde{r}(\ux) \leq \lambda \cdot R_0,
\\
T + \kappa_0 + \frac{\nu}{\pi} \log\tilde{r}(\ux), &
2 \lambda \cdot R_0 \leq \tilde{r}(\ux) \leq \frac{\iota_0}{4}, \\
T + \Ima h(0) - \frac{b}{\pi} \log\tilde{r}(\ux), &
2 \iota_0^{-1} \leq \tilde{r}(\ux) \leq r_{\lambda}, \\
1, &
\tilde{r}(\ux) \geq 2 r_{\lambda}.\\
 \end{cases}	
\end{align*}
\end{enumerate}

Next, we describe the $C^{k,\alpha}$-regularity scale.  
\begin{definition}[Local regularity]
\label{d:local-regularity} Let $(M^n,g)$ be a Riemannian manifold. Given $r, \epsilon>0$, $k\in\dN$, $\alpha\in(0,1)$,  $(M^n,g)$ is said to be $(r,k+\alpha,\epsilon)$-regular at $x\in M^n$ if   $g$ is at least $C^{k, \alpha}$ in $B_{2r}(x)$  such that the following holds: let $(\widehat{B_{2r}(x)},\hat{g},\hat{x})$ be the Riemannian universal cover of $B_{2r}(x)$, then $B_r(\hat{x})$ is diffeomorphic to a Euclidean disc $\dD^n$ such that for any $1\leq i,j\leq n$,  \begin{align*}
|\hat{g}_{ij}-\delta_{ij}|_{C^0(B_r(\hat{x}))}+\sum\limits_{|m|\leq k} r^{|m|}\cdot|\p^m \hat{g}_{ij}|_{C^0(B_r(\hat{x}))} +  r^{k+\alpha}[\hat{g}_{ij}]_{C^{k,\alpha}(B_r(\hat{x}))} < \epsilon,
\end{align*}
where $m$ is a multi-index, and the last term is the H\"older semi-norm.
\end{definition}

\begin{definition}
[$C^{k,\alpha}$-regularity scale] Let $(M^n,g)$ be a smooth Riemannian manifold. The $C^{k,\alpha}$-regularity scale $r_{k,\alpha}(x)$ at $x\in M^n$  is defined to be
the supremum of all $r>0$ such that $M^n$  is $(r, k+\alpha, 10^{-9})$-regular at $x$.
 \end{definition}
 \begin{remark}
 Note that $r_{k,\alpha}$ is 1-Lipschitz on any Riemannian manifold $(M^n,g)$, i.e., \begin{align}\label{e:1-Lipschitz-666}
|r_{k,\alpha}(x) - r_{k,\alpha}(y)|	
\leq d_{g}(x,y), \quad \forall x,y\in M^n.\end{align} 
\end{remark}

Let $\cS_b$ be the subset of $\cM_{\lambda}$ which consists of a small annular region in $\cK$ centered around the $\I_b^*$-fiber, the neck region $\fN$, and the ALG$^*$ manifold $X$.
The following proposition gives the regularity scale estimates and bubble limits of $g_\lambda$ on $\cS_b$.
\begin{proposition}
\label{p:regularity-scale-I_v*}
Let  $\fs$ be a smooth function that satisfies
\begin{align}
\fs(\bx) =
\begin{cases}
\tilde{\lambda} \cdot (\mathfrak{L}_T(\bx))^{\frac{1}{2}} \cdot \fr(\bx), &  \tilde{d}^Q(\bx, p_m) \geq 2\iota_0 \\
   & \text{for all}\ 1\leq m\leq 2\nu+2b,
\\
&
\\
\tilde{\lambda}\cdot \lambda \cdot \fd(\bx), & \tilde{d}^Q(\bx, p_m) \leq \frac{1}{4}\iota_0\\
 & \text{for some}\ 1\leq m\leq 2\nu+2b.
\end{cases}
\end{align}
Then the following properties hold.

\begin{enumerate}
\item

 Given $k\in\mathbb{N}$ and $\alpha\in(0,1)$,
there exists $v_0=v_0(k, \alpha)$ such that for any sufficiently small parameter $\lambda\ll1$ and $\bx\in \cS_b$, the $(k,\alpha)$-regularity scale $r_{k,\alpha}$ at $\bx$ satisfies
\begin{equation}
v_0^{-1}\leq \frac{r_{k,\alpha}(\bx)}{\fs(\bx)}\leq v_0.
\label{e:regularity-scale-lower-bound}
\end{equation}

\item There is a uniform constant $C_0>0$ such that
for every $\lambda\ll1$ and $\bx\in \cS_b$, we have
\begin{align*}
C_0^{-1} \leq  \frac{\fs(\by)}{\fs(\bx)}\leq C_0	, \quad \by\in B_{\fs(\bx)/4}(\bx).
\end{align*}

\item Let $\lambda_j\to0$ be a sequence, and let $\bx_j \in \cS_b$ be  a sequence of reference points. Then the rescaled spaces $(\cS_b, \fs(\bx_j)^{-2}\cdot g_{\lambda_j}, \bx_j)$ converges in the Gromov-Hausdorff topology to one of the following spaces as $\lambda_j\to0$:
\begin{itemize}
\item
the Taub-NUT space $(\dC^2, g^{TN})$ and the $\ALG_{\nu}^*$ gravitational instanton $(X, g^X)$,
\item
the flat manifolds $\dR^3$, $\dR^2\times S^1$,
$\dR^2$, and 
the flat cone $\dR^2/\dZ_2$,
\item
$\PP^1$ equipped with the McLean metric $d_{ML}$ with bounded diameter.
\end{itemize}
\end{enumerate}
\end{proposition}

\begin{proof}
We will prove \eqref{e:regularity-scale-lower-bound}
by contradiction.
Suppose that there does not exist a uniform constant $v_0$ with respect to fixed constants $k\in\dZ_+$ and $\alpha\in(0,1)$. That is, there are a sequence $\lambda_j\to0$ and  a sequence of points $\bx_j\in \cS_b$ such that
\begin{align}
\frac{r_{k,\alpha}(\bx_j)}{\fs(\bx_j)}\to 0 \quad \text{or}\quad \frac{r_{k,\alpha}(\bx_j)}{\fs(\bx_j)}\to \infty.\label{e:scale-contradiction}
\end{align}
Let us work with the rescaled sequence $(\cS_b,\hat{g}_{\lambda_j},\bx_j)$ with $\hat{g}_{\lambda_j}\equiv \fs(\bx_j)^{-2} \cdot g_{\lambda_j}$ as $\lambda_j\to0$. In the proof, we will show that $C^{k,\alpha}$-regularity scale at $\bx_j$ with respect to $\hat{g}_{\lambda_j}$ is uniformly bounded from above and below as $\lambda_j\to 0$ which contradicts \eqref{e:scale-contradiction}.
We will derive a contradiction in each of the following cases depending upon the location of $\bx_j$. Denote $\ux_j\equiv \pi(\bx_j) \in Q/\ZZ_2$ for any $\bx_j \in \fN$.

\vspace{0.2cm}
\noindent
{\bf Case ($\I$):} there exists a constant $\sigma_0\geq 0$ such that $\tilde{r}(\bx_j) \cdot \lambda_j^{-1}\to \sigma_0$ as $j\to \infty$. Let us consider the $\sigma_0\leq R_0$ case first.
By definition, we have 
$\fs(\bx_j) = \tilde{\lambda}_j \cdot \lambda_j \cdot R_0$.
We consider the rescaled metric 
\begin{equation}
\hat{g}_{\lambda_j} \equiv (\tilde{\lambda}_j \cdot \lambda_j \cdot R_0) ^{-2} g_{\lambda_j}.
\end{equation}
By the gluing construction, we have that, for any $k\in\dZ_+$, 
\begin{equation}
(\cS_b, \hat{g}_{\lambda_j}, \bx_j)\xrightarrow{C^k} 	(X, R_0^{-2} \cdot g^{X}, \bx_{\infty}).
\end{equation} Notice that the ALG$^*$ gravitational instanton $(X,R_0^{-2}\cdot g^X)$ is a Ricci-flat but non-flat space, which implies that for any $k\in\dZ_+$ and $\alpha\in(0,1)$  there exists a constant $v_0>0$ such that $\frac{2}{v_0} \leq r_{k,\alpha}(\bx_{\infty}) \leq \frac{v_0}{2}$. 
Therefore, for any $k\in\dZ_+$ and $\alpha\in(0,1)$,  $v_0^{-1} \leq r_{k,\alpha}(\bx_j) \leq v_0$ holds with
 respect to  $\hat{g}_{\lambda_j}$, 
 which contradicts \eqref{e:scale-contradiction}. Therefore, the proof in the case  $\sigma_0 \leq  R_0$ is complete.
The proof in the $\sigma_0 > R_0$ case is the same.

\vspace{0.2cm}
\noindent  
{\bf Case ($\II_1$):}  there exists some constant $\gamma_0>0$ such that
\begin{align}
\lambda_j^{-1}\cdot\tilde{d}^Q(\ux_j, p_m) \cdot T^{\flat}_j \to \gamma_0  \quad \text{as}\ j\to \infty,
\end{align}
where $T_j^{\flat} \equiv \frac{2 \nu + 1}{2 \pi}\log(\frac{1}{\lambda_j})$.
We first assume $\gamma_0 \leq 1$.
  By definition,  $\fs(\bx_j) = \tilde{\lambda}_j\cdot \lambda_j\cdot (T^{\flat}_j)^{-\frac{1}{2}}$. Recall that the metric $g_{\lambda_j}$ satisfies $g_{\lambda_j}=\tilde{\lambda}_j^2 \cdot g^{\mathcal{N}}$ near $\bx_j$ and  by \eqref{e:neck-metric},
\begin{align}
g^{\mathcal{N}}=\lambda_j^2\cdot (G_{\lambda_j}\cdot g^Q + G_{\lambda_j}^{-1}\Theta^2 ),
\end{align}
where \begin{align*}
|G_{\lambda_j}(\underline{\bx}) - T^{\flat}_j - \frac{1}{2 d^Q(\underline{\bx}, p_m)}| \le C  \quad \text{for}\ d^Q(\underline{\bx}, p_m)\leq r_0\equiv \frac{1}{4}\InjRad_{g^Q}(Q). 
\end{align*}
Let $(u_1,u_2,u_3)$ be a fixed coordinate system in $B_{r_0}(p_m)$ with respect to the metric $g^Q$.  Consider the rescaled metric 
   $\hat{g}_{\lambda_j}\equiv \fs(\bx_j)^{-2} \cdot g_{\lambda_j}$ and the rescaled coordinates centered at $p_m=(p_{1,m}, p_{2,m}, p_{3,m})\in P$,
   \begin{align*}
  	(\hat{u}_1, \hat{u}_2, \hat{u}_3) \equiv  T^{\flat}_j  (u_1-p_{1,m}, u_2-p_{2,m}, u_3-p_{3,m}).
   \end{align*}
 Then by the explicit computations, $(\cS_b, \hat{g}_{\lambda_j}, \bx_j)$ $C^k$-converges for any $k\in\dZ_+$, to the Taub-NUT space $(\dC^2, g^{TN}, \bm{x}_{\infty})$,  where the Taub-NUT metric $g^{TN}$ can be written explicitly in terms of the  Gibbons-Hawking ansatz
\begin{equation}
g^{TN} = V_0 g^{\dR^3} + V_0^{-1}\Theta^2, \ V_0=1+(2r)^{-1},
\end{equation}
and $r$ is the Euclidean distance to the origin of $\dR^3$. Therefore, there exists a constant $v_0$ such that $v_0^{-1}\leq r_{k,\alpha}(\bx_j)\leq v_0$
with respect to the rescaled metric $\hat{g}_{\lambda_j}$. Rescaling back to $g_{\lambda_j}$, we find that the above estimate contradicts \eqref{e:scale-contradiction}.
This completes the proof under the assumption $\gamma_0\leq 1$. The proof in the case $\gamma_0>1$ is the same.  

\vspace{0.2cm}
\noindent
{\bf Case ($\II_2$):} for some $p_m\in P$, the points $\bx_j$ satisfy  
\begin{equation}
\lambda_j^{-1}\cdot\tilde{d}^Q(\underline{\bx}_j, p_m) \cdot T_j^{\flat} \to \infty \ \text{and} \ \lambda_j^{-1}\cdot\tilde{d}^Q(\underline{\bx}_j, p_m)  \to 0
\end{equation}  as  $j\to \infty$. In this case, by definition 
\begin{align}
\fs(\bx_j) = \tilde{\lambda}_j\cdot \lambda_j \cdot (T_j^{\flat})^{\frac{1}{2}} \cdot d_j,
\end{align}
where $d_j  \equiv d^Q(\underline{\bx}_j,p_m)$.
We will work with 
$\hat{g}_{\lambda_j} \equiv  \fs(\bx_j)^{-2}\cdot g_{\lambda_j}$ and the rescaled coordinates centered at $p_m=(p_{1,m},p_{2,m},p_{3,m})\in P$,
\begin{align*}
	(\hat{u}_1,\hat{u}_2,\hat{u}_3) \equiv d_j^{-1}\cdot (u_1-p_{1,m}, u_2-p_{2,m}, u_3-p_{3,m}),\end{align*}
where $(u_1,u_2,u_3)$ is a fixed coordinate system in $B_{r_0}(p_m)$.
One can verify 
\begin{align}
(\cS_b, \hat{g}_{\lambda_j}, \bx_j) \xrightarrow{GH}
(\dR^3, g^{\dR^3}, \bx_{\infty})
\end{align}
with $d^{\dR^3}(\bx_{\infty}, 0^3) = 1$. The detailed and explicit rescaling computations can be found in \cite{HSVZ} and \cite{CVZ}.
Moreover,  if we lift the metric
to the local universal cover around $\bx_j$, then a ball of definite size radius has uniformly bounded $C^{k,\alpha}$-geometry.
This implies $r_{k,\alpha}(\bx_j)\geq v_0 > 0$. The upper bound for $r_{k,\alpha}(\bx_j)$ follows from \eqref{e:1-Lipschitz-666} 
 and the calculation in Case ($\II_1$). We get a contradiction with \eqref{e:scale-contradiction} as rescaling back to $g_{\lambda_j}$.

Notice that $\dR^3$ is precisely the asymptotic cone of the Taub-NUT space.
 
\vspace{0.2cm}
\noindent
  {\bf Case ($\II_3$):} there is some constant $d_0$ such that   for some $p_m\in P$,
  \begin{align}
 \lambda_j^{-1}\cdot \tilde{d}^Q(\underline{\bx}_j,p_m) \to d_0 \quad \text{as}\ j \to\infty.
  \end{align}
By the definition of $\fs$, we have that
\begin{equation}
\fs(\bx_j) = \tilde{\lambda}_j\cdot \lambda_j \cdot (T_j^{\flat})^{\frac{1}{2}} \cdot d_0 \cdot (1+o(1)).
\end{equation}	
It suffices to work with the rescaled metric $(\tilde{\lambda}_j\cdot \lambda_j \cdot (T_j^{\flat})^{\frac{1}{2}} \cdot d_0)^{-2}\cdot g_{\lambda_j}$, still denoted by $\hat{g}_{\lambda_j}$, and prove that the regularity scale $r_{k,\alpha}(\bx_j)$ is uniform bounded from above and below. Then the contradiction arises.

Straightforward computations imply that $\hat{g}_{\infty} = d_0^{-2}\cdot g^Q$, which is a rescaling of the  flat base metric $g^Q$ on $\dR^2\times S^1$.
 Since $d^Q(P,0^*) \geq \underline{B}_0 \cdot \lambda_j^{-1} \to \infty$, it follows that the origin $0^2\in\dR^2$ translates to infinity and the $\dZ_2$-action limits to the identity. Therefore,
the rescaled limit is isometric to $\dR^2\times S^1$.
The collapsing keeps curvature uniformly bounded away from $P$. Then there is some uniform constant $v_0>0$ such that  $r_{k,\alpha}(\bx_j)\geq v_0 > 0$. The upper bound for $r_{k,\alpha}(\bx_j)$ follows from \eqref{e:1-Lipschitz-666}  and the calculation in Case ($\II_2$).

Notice that, $\dR^2\times S^1$ is the flat base of the metric $g^{\mathcal{N}}$.

\vspace{0.2cm}
\noindent
 {\bf Case ($\II_4$):} for some $p_m\in P$, we have \begin{align}\lambda_j^{-1}\cdot \tilde{d}^Q(\underline{\bx}_j,p_m)\to \infty \quad \text{and}\quad  \tilde{d}^Q(\underline{\bx}_j,p_m) \to 0 \quad \text{as} \ j\to\infty.\end{align}
Let us denote $d_j \equiv \lambda_j^{-1}\cdot \tilde{d}^Q(\underline{\bx}_j, p_m)$.
In this case, the definition of $\fs$ implies
\begin{align}\fs(\bx_j) = \tilde{\lambda}_j\cdot \lambda_j \cdot (T_j^{\flat}+\frac{1}{2 \pi}\log \frac{1}{d_j})^{\frac{1}{2}}\cdot d_j.\end{align}
We will prove that in terms of the rescaled metric $\hat{g}_{\lambda_j} \equiv \fs(\bx_j)^{-2}g_{\lambda_j}$, the regularity scale $r_{k,\alpha}(\bx_j)$ has a uniform lower bound and upper bound, which contradicts \eqref{e:scale-contradiction}.

The flat product metric $g^Q$ can be written as $g^Q = d x^2 + d y^2 + d \theta_2^2$ in coordinates.
We also rescale the above coordinate system of $\dR^2$ centered around $\underline{\bx}_j = (x_j, y_j)$ by letting
\begin{align}(\hat{x}, \hat{y})\equiv d_j^{-1}\cdot (x - x_j,y - y_j).\end{align}
Explicit tensorial computations show that
\begin{align}
(\cS_b, \hat{g}_{\lambda_j}, \bx_j) \xrightarrow{GH} (\dR^2, g^{\dR^2}, 0^2),
\end{align}	
 where the Euclidean metric $g^{\dR^2}$ has the expression $g^{\dR^2}\equiv d\hat{x}_{\infty}^2 + d\hat{y}_{\infty}^2$.
 By assumption,  
$d^Q(\underline{\bx}_j, 0^*)/d_j \to \infty$,  which implies that the origin  $0^2\in \dR^2$ translates to infinity and hence the $\dZ_2$-action limits to the identity as $j\to\infty$.

The finite set $P$ converges to a single point $p_0\in \dR^2$ and $d^{\dR^2}(p_0,0^2)=1$. The above collapsing keeps curvature uniformly bounded away from the point $p_0$. So there is some uniform constant $v_0>0$ such that  $r_{k,\alpha}(\bx_j)\geq v_0 > 0$. The upper bound for $r_{k,\alpha}(\bx_j)$ follows from \eqref{e:1-Lipschitz-666} and the calculation in Case ($\II_3$).

\vspace{0.2cm}
\noindent
 {\bf Case ($\III$):} there exists some constant $d_0$ such that  
\begin{align}
\tilde{d}^Q(\underline{\bx}_j, P) \geq  d_0>0,\quad  \tilde{r}(\ux_j) \cdot \lambda_j^{-1}\to \infty, \quad  L_j\equiv \mathfrak{L}_{T_j}(\ux_j)\to +\infty.
 \end{align}	
There are the following subcases to analyze.

First, assume that $\tilde{r}(\ux_j) \to 0$  as $j\to \infty$.	
In this case, by definition, 
\begin{align}
\fs(\bx_j) = \tilde{\lambda}_j \cdot L_j^{\frac{1}{2}}\cdot \tilde{r}_j.
\end{align}
We will prove that under the rescaling $\hat{g}_{\lambda_j}= (\fs(\bx_j))^{-2}\cdot g_{\lambda_j}$ and  $(\hat{x}, \hat{y}) = \tilde{r}_j^{-1} \cdot (\tilde{x}, \tilde{y})$,  the  convergence 
\begin{align}
(\cS_b, \hat{g}_{\lambda_j}, \bx_j)	\xrightarrow{GH} (\dR^2/\dZ_2, d^{\dR^2/\dZ_2}, \ux_{\infty})
\end{align}
holds,
where $d^{\dR^2/\dZ_2}(\ux_{\infty}, 0^2)  = 1$, and the flat metric on $\dR^2/\dZ_2$ can be written in terms of the limit coordinate system of $(\hat{x}, \hat{y})$.
Notice that $\dR^2/\dZ_2$ is the asymptotic cone of the $\ALG_{\nu}^*$ space $(X, g^{X},\bm{\omega}^X)$.
Moreover, we will show that the rescaled metrics $\hat{g}_{\lambda_j}$
has uniformly bounded curvature away from the cone tip. This suffices to produce the desired contradiction because the upper bound on $r_{k,\alpha}(\bx_j)$ follows from \eqref{e:1-Lipschitz-666} and the calculation in Case ($\II_4$).

To prove the above claim,
let us choose a domain
\begin{align}
\mathcal{U}_{\xi_j} \equiv \{\bx\in \fN\subset \cS_b \subset \cM_{\lambda_j} |\xi_j^{-1}\leq \hat{r}(\bx)\leq \xi_j\} 	
\end{align}
for a sequence $\xi_j$ that satisfies $\lim\limits_{j\to \infty} \frac{\xi_j}{L_j} = 0$.
Then for any $\bx_j\in \mathcal{U}_{\xi_j}$,
\begin{align}
\frac{\widetilde{L}_{2\nu}(\ux_j)}{L_j} = 1 + o(1)\quad \text{as}\ j\to\infty.\end{align}
By explicit tensorial computations on the Gibbons-Hawking metric $g^{\mathcal{N}}$, we can check that the $\dZ_2$-covering of $(\mathcal{U}_{\xi_j}, \hat{g}_{\lambda_j})$
will smoothly converge to the flat metric  $d \hat{x}_{\infty}^2 + d\hat{y}_{\infty}^2$ on $\dR^2$, where $(\hat{x}_{\infty},\hat{y}_{\infty})$
is the limit of $(\hat{x}, \hat{y})$. Also notice that the limiting reference point $\ux_{\infty}$ satisfies $d^{\dR^2}(\ux_{\infty}, 0^2) =1$. Then the $\dZ_2$-quotient metric $\hat{g}_{\lambda_j}$
converges to the flat metric on $\dR^2/\dZ_2$.

Next, we consider the case  $\tilde{r}(\ux_j)\to d_0' > 0$. By definition, 
\begin{equation}
\fs(\bx_j) = \tilde{\lambda}_j  \cdot T_j^{\frac{1}{2}} \cdot d_0' \cdot (1+o(1))
\end{equation}
as $j \to \infty$. It suffices to work with the rescaled metric
$(\tilde{\lambda}_j \cdot T_j^{\frac{1}{2}} \cdot d_0')^{-2}\cdot g_{\lambda_j}$, still denoted by $\hat{g}_{\lambda_j}$,
and we can show that the regularity scale $r_{k,\alpha}(\bx_j)$ has a uniform lower bound. Moreover,
the rescaled limit in this case is isometric to $\dR^2/\dZ_2$ as well. We skip the detailed computations since the arguments are the same.

In the last case, we will consider the case when $\bx_j$ satisfies $\tilde{r}(\ux_j)\to\infty$ and $L_j\to \infty$. In the proof, we still use the rescaled metric $\hat{g}_{\lambda_j}= (\fs(\bx_j))^{-2}\cdot g_{\lambda_j}$ and the rescaled coordinates $(\hat{x}, \hat{y}) = \tilde{r}_j^{-1} \cdot (x,y)$. The computations are the same. We only mention that, as $\tilde{r}_j = \tilde{r}(\ux_j)$ becomes very large, one can obtain the rescaled limit $\dR^2/\dZ_2$ as long as $L_j\to \infty$.

When $\tilde{r}_j$ is sufficiently large such that $L_j\to L_0>0$ as $j\to\infty$, we will obtain another rescaled limit. This becomes Case (IV).

\vspace{0.2cm}
\noindent
 {\bf Case ($\IV$):} there is some constant $L_0>0$ such that $L_j\to L_0>0$. In this case, 
\begin{equation}
\fs(\bx_j) \equiv \tilde{\lambda}_j\cdot L_j^{\frac{1}{2}} \cdot  \tilde{r}_j =\tilde{\lambda}_j \cdot \tilde{r}_j \cdot  L_0 \cdot (1+o(1))
\end{equation}
as $j\to\infty$. In the meantime, notice that 
\begin{equation}
L_j = (T_j + \Ima h(0)- \frac{b}{\pi}\log\tilde{r}_j)\cdot (1+o(1))
\end{equation}
as $j\to\infty$. It is easy to verify that
$\fs(\bx_j)$ is a bounded constant. Then the rescaled limit is the McLean metric on $\PP^1$. Moreover, the convergence keeps curvature uniformly bounded away from the singular fiber.

The above covers all the points on $\cS_b $  which completes the proof.
\end{proof}

\section{Perturbation to hyperk\"ahler metrics}

\label{s:perturbation}

In this subsection, we will glue an ALG$^*$ graviational instanton into a region near an $\I_b^*$-fiber of an elliptic K3 surface $\pi_{\cK}: \cK\to \PP^1$. 
For our purpose, it suffices to assume that the singular fibers of $\pi_{\cK}$ consists of an $\I_b^*$-fiber 
for some $1 \leq b \leq 14$ and $\I_1$ fibers of number $(18-b)$. Following the notations in Section 6 of \cite{CVZ}, we denote by $\cS_b$ the subset of $\cM_{\lambda}$ which consists of a small annular region in $\cK$ centered around the $\I_b^*$-fiber, the neck region $\fN$, and the ALG$^*$ manifold $X$. We denote by $\mathcal{S}_{\I_1}$ the subset of $\cM_{\lambda}$ which consists of small annular regions in $\cK$ centered around $\I_1$-fibers and Ooguri-Vafa manifolds. Let $\mathcal{R}_{\lambda}$ be the regular region in $\cK$. We will prove that the glued manifold $\cM_{\lambda}$ admits collapsing hyperk\"ahler metrics with prescribed behaviors. 
 In the following weighted analysis,
the weight function $\rho$ as a global smooth function on $\cM_{\lambda}$ is defined as follows,
\begin{align}
\label{d:weight}
\rho(\bx)
\equiv
\begin{cases}
\fs(\bx), & \bx \in \cS_b,
\\
\fs_{1}(\bx), & \bx \in \cS_{\I_1},
\\
1, & \bx \in \cR_{\lambda},
\end{cases}
\end{align}
where $\fs_1$ is the canonical scale function defined in Section 6.3 of \cite{CVZ}. 
With respect to the weight function $\rho$, we will define the following weighted H\"older norms.
\begin{definition}\label{d:weighted-space} For any fixed  parameter $\lambda\ll1$, let $g_{\lambda}$ be the approximately hyperk\"ahler metric defined on the glued manifold $\cM_{\lambda}$. Let $U\subset\M_{\lambda}$ be a compact subset. Then the  weighted H\"older norm of a tensor field $\chi\in T^{r,s}(U)$ of type $(r,s)$ is defined as follows:
\begin{enumerate}

\item The weighted $C^{k, \alpha}$-seminorm of $\chi$ is defined by
\begin{align*}
[\chi]_{C_{\mu}^{k,\alpha}}(\bx) & \equiv  \sup\Big\{\rho^{k + \alpha - \mu}(\bx) \cdot\frac{|\nabla^k\hat{\chi}(\hat{\bx})- \nabla^k\hat{\chi}(\hat{\by})|}{(d_{\hat{g}_{\lambda}}(\hat{\bx},\hat{\by}))^{\alpha}}  \ \Big| \  \hat{\by}\in B_{r_{k,\alpha}(\bx)}(\hat{\bx})\Big\},
\\
[\chi]_{C_{\mu}^{k,\alpha}(U)} & \equiv \sup\Big\{[\chi]_{C_{\mu}^{k,\alpha}}(\bx)\Big|\bx \in U\Big\},
\end{align*}
where $r_{k,\alpha}(\bx)$ is the $C^{k,\alpha}$-regularity scale at $\bx$, $\hat{\bx}$ denotes a lift of $\bx$ to the universal cover of $B_{2r_{k,\alpha}(\bx)}(\bx)$,
the difference of the two covariant derivatives is defined in terms of parallel translation in $B_{r_{k,\alpha}(\bx)}(\hat{\bx})$,
and $\hat{\chi}$, $\hat{g}_{\lambda}$ are the lifts of $\chi$, $g_{\lambda}$ respectively.

\item The weighted $C^{k,\alpha}$-norm of $\chi$ is defined by
\begin{align*}
\|\chi\|_{C_{\mu}^{k,\alpha}(U)}
  \equiv \sum\limits_{m=0}^k\Big\|\rho^{m - \mu}  \cdot\nabla^m \chi\Big\|_{C^0(U)} + [\chi]_{C_{\mu}^{k,\alpha}(U)}.
\end{align*}
\end{enumerate}

\end{definition}

Now let us briefly describe the perturbation scheme to produced hyperk\"ahler triples from the approximate triples constructed in Section \ref{s:gluing-construction}. This original characterization is due to Donaldson \cite{Donaldson}, which has also been used in \cite{CCIII, CVZ, Foscolo2020, FLS, HSVZ}.  
Let $M^4$ be an oriented $4$-manifold with a volume form $\dvol_0$. A triple of closed $2$-forms 
 $\bm{\omega}=(\omega_1,\omega_2,\omega_3)$ 
is said to be \textit{definite} if the matrix $Q=(Q_{ij})$ defined by 
$\frac{1}{2}\omega_i\wedge\omega_j=Q_{ij}\dvol_0$ is positive definite. A definite triple $\bm{\omega}$ is called a {\it hyperk\"ahler triple} if $Q_{ij}=\delta_{ij}$. 
Given a definite triple $\bm{\omega}$, the associated volume form is defined as
$\dvol_{\bm{\omega}} \equiv (\det(Q))^{\frac{1}{3}} \dvol_0$,
and we denote by 
$Q_{\bm{\omega}}\equiv (\det(Q))^{-\frac{1}{3}}Q$
the normalized matrix with unit determinant. 
Every definite triple $\bm{\omega}$ determines a Riemannian metric $g_{\bm{\omega}}$ such that each $\omega_j$, $j\in\{1,2,3\}$, is self-dual with respect to $g_{\bm{\omega}}$ and $\dvol_{g_{\bm{\omega}}} = \dvol_{\bm{\omega}}$.

Suppose that we have a closed definite triple $\bm{\omega}$ on $\cM_{\lambda}$. We want to find a triple of closed $2$-forms $\bm{\theta}=(\theta_1, \theta_2, \theta_3)$ such that $
\underline{\bm{\omega}}\equiv\bm{\omega}+\bm{\theta}$ is an actual hyperk\"ahler triple on $\cM_{\lambda}$ satisfying
\begin{equation}
\label{e:hk1}
\frac{1}{2}(\omega_i+\theta_i)\wedge(\omega_j+\theta_j)=\delta_{ij} \dvol_{\bm{\omega}+\bm{\theta}},
\end{equation} 
 which is equivalent to
\begin{equation}
\frac{1}{2}(\omega_i\wedge\omega_j + \omega_i\wedge\theta_j + \omega_j\wedge\theta_i + \theta_i\wedge \theta_j) =\frac{1}{6}\delta_{ij}\sum\limits_{k=1}^3\Big(\omega_k^2 + \theta_k^2+ 2\omega_k\wedge \theta_k\Big).\label{e:expansion}
\end{equation}
 Writing 
$\bm{\theta}=\bm{\theta}^++\bm{\theta}^-$ with $*_{g_{\bm{\omega}}}\bm{\theta}^{\pm}=\pm\bm{\theta}^{\pm}$,   we define the matrices $A=(A_{ij})$ and $S_{\bm{\theta}^-}=(S_{ij})$ by  
\begin{align}
\label{e:SD-matrix}\theta^+_i = \sum\limits_{j=1}^3 A_{ij} \omega_j, \quad \frac{1}{2}\theta_i^{-}\wedge\theta_j^{-}=S_{ij}\dvol_{\bm{\omega}},\ 1\leq i\leq j\leq 3.
\end{align} Then \eqref{e:expansion} is equivalent to 
\begin{equation}
\TF(Q_{\bm{\omega}}A^T + Q_{\bm{\omega}}A + AQ_{\bm{\omega}}A^T) = \TF(-Q_{\bm{\omega}}-S_{\bm{\theta}^-}),\label{e:lftf}
\end{equation}
where $\TF(B) \equiv B - \frac{1}{3}\Tr(B)\Id$ for a $3\times 3$ real matrix $B$, and $Q_{\bm{\omega}}$ is the $3\times 3$ real matrix such that $\det (Q_{\bm{\omega}})=1$ and
\begin{equation}
\frac{1}{2} \omega_i \wedge \omega_j = (Q_{\bm{\omega}})_{ij} \dvol_{\bm{\omega}}.
\end{equation}
 Then observe that a solution of 
\begin{align}
\begin{split}
& d^+ \bm{\eta} + \bm{\xi} = \mathfrak{F}_0\Big(\TF(-Q_{\bm{\omega}}-S_{d^{-}\bm{\eta}})\Big), \\
& d^* \bm{\eta} = 0,\quad \quad  \bm{\eta}\in\Omega^1(\cM_{\lambda})\otimes\dR^3,\  
\bm{\xi}\in\mathcal{H}_{g_{\bm{\omega}}}^+(\cM_{\lambda})\otimes\dR^3,
\end{split}
\label{e:hkt-system}
\end{align} 
is also a solution of \eqref{e:lftf}. Here $\mathfrak{F}_0$
denotes the local inverse near zero of \begin{align*}\mathfrak{G}_0:\mathscr{S}_0(\dR^3) \to \mathscr{S}_0(\dR^3), \quad A \mapsto \TF(Q_{\bm{\omega}}A^T + AQ_{\bm{\omega}}+ AQ_{\bm{\omega}}A^T).               \end{align*}
 on the space of trace-free symmetric $3\times 3$-matrices $\mathscr{S}_0(\dR^3)$, and $d^{\pm}\bm{\eta}$ is the self-dual or anti-self-dual part of  $d \bm{\eta} = \bm{\theta} - \bm{\xi}$, respectively. 
The linearization of the elliptic system \eqref{e:hkt-system} at $\bm{\eta} = 0$ is given by
$\mathscr{L} = (\mathscr{D} \oplus \Id)\otimes \dR^3 : (\Omega^1(\cM_{\lambda}) \oplus \mathcal{H}_{g_{\bm{\omega}}}^+(\cM_{\lambda}) )\otimes \dR^3 \longrightarrow ( \Omega^0(\cM_{\lambda}) \oplus \Omega^2_+(\cM_{\lambda}))\otimes \dR^3$,
 where
\begin{equation*}
\mathscr{D}\equiv d^*+ d^+: \Omega^1(\cM_{\lambda}) \longrightarrow (\Omega^0(\cM_{\lambda}) \oplus \Omega^2_+(\cM_{\lambda})).
\end{equation*}
For any sufficiently small   $\lambda\ll1$,  we  will solve the elliptic system
\eqref{e:hkt-system}. 
 The proof of the existence of hyperk\"ahler triples requires the following version of implicit function theorem.
\begin{lemma} \label{l:implicit-function}
Let $\mathscr{F}:\mathfrak{A} \to \mathfrak{B}$ be a map between two Banach spaces with
\begin{equation}
\mathscr{F}(\bx)=\mathscr{F}(\bo)+\mathscr{L}(\bx)+\mathscr{N}(\bx),\end{equation} where the operator $\mathscr{L}:\mathfrak{A}\to\mathfrak{B}$ is linear and $\mathscr{N}(\bo)=\bo$. Assume that
\begin{enumerate}
\item $\mathscr{L}$ is an isomorphism with $\| \mathscr{L}^{-1} \| \leq C_L$ for some $C_L>0$.

\item there are constants $r > 0$ and $C_N>0$ such that:
\begin{enumerate}
\item $r < (10 C_L \cdot C_N)^{-1}$,
\item  $\| \mathscr{N}(\bx) - \mathscr{N}(\by) \|_{\mathfrak{B}} \leq C_N\cdot ( \|\bx\|_{\mathfrak{A}} + \|\by\|_{\mathfrak{A}} ) \cdot  \| \bx - \by \|_{\mathfrak{A}} $ for all $x,y\in B_r(\bo)\subset {\mathfrak{A}}$,

\item $\| \mathscr{F}(\bo) \|_{\mathfrak{B}} \leq \frac{r}{10C_L}$.
\end{enumerate}
\end{enumerate}
Then $\mathscr{F}(\bx)=\bo$ has a unique solution $\bx\in \mathfrak{A}$   such that
$\|\bx\|_{\mathfrak{A}} \leq 2C_L  \|\mathscr{F}(\bo)\|_{\mathfrak{B}}$.

\end{lemma}

To apply the implicit function theorem, first we fix two Banach spaces
\begin{align*}
\mathfrak{A}  \equiv \Big(C_{\mu}^{1,\alpha}(\mathring{\Omega}^1(\cM_{\lambda}))\oplus \mathcal{H}^+(\cM_{\lambda})\Big)\otimes\dR^3,
\quad
\mathfrak{B}  \equiv\Big(C_{\mu-1}^{0,\alpha}(\Omega^2_+(\cM_{\lambda}))\Big)\otimes\dR^3,
\end{align*}
where $\mu \in (-1, 0)$, $\alpha \in (0, 1)$,
and
$\mathring{\Omega}^1(\cM_{\lambda}) \equiv \{\eta \in \Omega^1(\cM_{\lambda}) | \ d^*\eta = 0\}$.
 The following error estimate is an immediately corollary of Lemma \ref{l:approximate-triple}.
\begin{corollary}\label{c:weighted-error}
There exists $C_0>0$ independent of the parameters $\lambda$ and $\ft$ such that 
\begin{align*}
\|\mathscr{F}(\bo)\|_{\fB}\leq C_0 \cdot (\lambda^2 + \tilde{\lambda} \cdot \ft^3) \cdot \tilde{\lambda}^{-\mu+1}\cdot \ft^{-\mu-1} \cdot V(\sigma^{-1})^{\frac{-\mu-1}{2}} + C_0 \cdot \lambda^2 \cdot \tilde{\lambda}^2.
\end{align*}
\end{corollary}
\begin{proof} Let $\mathfrak{N}_r\equiv\{\bx\in\fN| r\leq \tilde{r}(\bx)\leq 2r\}$. Then by Lemma \ref{l:approximate-triple},
\begin{align*}
& \|Q_{\bm{\omega}}-\Id\|_{C_{\mu-1}^{0,\alpha}(\mathfrak{N}_{\ft})}\leq   C_0 \cdot (\lambda^2 + \tilde{\lambda} \cdot \ft^3) \cdot \tilde{\lambda}^{-\mu+1}\cdot \ft^{-\mu-1} \cdot V(\sigma^{-1})^{\frac{-\mu-1}{2}},	
\\
& \|Q_{\bm{\omega}}-\Id\|_{C_{\mu-1}^{0,\alpha}(\mathfrak{N}_{r_{\lambda}})}\leq   C_0 \cdot \lambda^2 \cdot \tilde{\lambda}^2.
\end{align*}
 On the other hand, the error estimate near an $\I_1$-fiber is much smaller. In fact, by Theorem 4.4 of \cite{GW} (see also \cite[Proposition 8.2]{CVZ}),  
 \begin{align}
\|Q_{\bm{\omega}}-\Id\|_{C_{\mu-1}^{0,\alpha}(T_{\delta_0,2\delta_0}(\mathcal{S}_{\I_1}))}\leq C_1 \cdot e^{-C_2\cdot \tilde{\lambda}^{-1}}
\end{align}
for some constants $C_1>0$, $C_2>0$ independent of $\tilde{\lambda}$ (and hence $\lambda$), where $T_{\delta_0,2\delta_0}(\mathcal{S}_{\I_1})$ is an annular neighborhood of definite size $\delta_0>0$ independent of $\tilde{\lambda}$. 
\end{proof}
We also need the weighted estimate on the nonlinear errors.
\begin{lemma}
[Nonlinear estimate]\label{l:nonlinear-term}
There exists some constant $K_0>0$ independent of $\lambda$ such that for any $v_1,v_2\in B_1(\bm{0})\subset \fA$, we have that
\begin{align*}
\| \mathscr{N}_{\lambda}(v_1) - \mathscr{N}_{\lambda}(v_2) \|_{\mathfrak{B}} \leq K_0 \cdot (\tilde{\lambda} \cdot \lambda)^{\mu - 1}  (T^{\flat})^{\frac{1 - \mu}{2}}
 (\|v_1\|_{\mathfrak{A}} + \|v_2\|_{\mathfrak{A}} ) \cdot \|v_1 - v_2\|_{\mathfrak{A}}.
\end{align*}
\end{lemma}
\begin{proof}
For any $v_1, v_2\in B_1(\bo)\subset \fA$, by explicit computations,
\begin{align*}
 |\mathscr{N}_{\lambda}(v_1) - \mathscr{N}_{\lambda}(v_2) |
\leq K_0 \cdot ( |d^{-}\eta_1 | + |d^{-}\eta_2|) \cdot |d^{-}(\eta_1 - \eta_2)|,
\end{align*}
where $K_0 > 0$ is independent of $\lambda$.
 Multiplying by the weight function $\fs(\bx)^{- \mu +1}$, we have that
\begin{align*}
&\ \fs(\bx)^{-\mu +1}\cdot | \mathscr{N}_\lambda(v_1) - \mathscr{N}_\lambda(v_2)|
\\ \leq &\  K_0\cdot \fs(\bx)^{- \mu +1} \cdot ( |d^{-} \eta_1| + |d^{-} \eta_2|) \cdot |d^{-} (\eta_1 - \eta_2)|.
\end{align*}
By definition, the scale function $\fs(\bx)$ achieves the minimum $\tilde{\lambda} \cdot \lambda \cdot (T^{\flat})^{- \frac{1}{2}}$ when $d^Q(\ux, p_m)\leq   (T^{\flat})^{-1}$ for some $\ 1\leq m \leq 2 \nu + 2 b$. Then we have that
\begin{align*}
&\ \Vert \mathscr{N}_\lambda(v_1) - \mathscr{N}_\lambda(v_2) \Vert_{ C^0_{\mu+1}(\cM_{\lambda})}\nonumber\\
\leq &\ K_0 \cdot (\tilde{\lambda}\cdot \lambda)^{\mu-1}  (T^{\flat})^{\frac{1 - \mu}{2}}  \Big(  \Vert v_1 \Vert_{ C^1_{\mu}(\cM_{\lambda})}  +
\Vert v_2 \Vert_{ C^1_{\mu}(\cM_{\lambda})} \Big)\cdot \Big(
\Vert v_1 - v_2 \Vert_ { C^1_{\mu}(\cM_{\lambda})}\Big).
\end{align*}
By similar computations,
we also have the desired estimate for the $C^{0,\alpha}$-seminorm. This completes the proof.
\end{proof}

The following is the main ingredient to carry out the perturbation.
\begin{proposition}[Weighted linear estimate]	
\label{p:injectivity-for-L} Let $\cM_{\lambda}$ be the glued manifold with a family of approximately hyperk\"ahler metrics $g_{\lambda}$. Then there exists $C>0$, independent of $\lambda$, such that for every self-dual $2
$-form $\xi^+\in\mathfrak{B}
$, there exists a unique pair
$(\eta,\bar{\xi}^+)\in \mathfrak{A}$
such that for some  $\mu\in(-1,0)$ and $\alpha\in(0,1)$, \begin{align}
&\mathscr{L}_{\lambda}(\eta, \bar{\xi}^+) = \xi^+,
\\
&\|\eta\|_{C_{\mu}^{1,\alpha}(\cM_{\lambda})}+\|\bar{\xi}^+\|_{C_{\mu-1}^{0,\alpha}(\cM_{\lambda})}\leq
 C  \|\xi^+\|_{C_{\mu-1}^{0,\alpha}(\cM_{\lambda})}.\label{e:L-uniform-estimate}
\end{align}
\end{proposition}

The proof is very similar to the proof of Proposition 8.7 in \cite{CVZ} which follows from a contradiction argument and applying various Liouville theorems on the blow-up limits. We omit the details and only mention the outline.
\begin{enumerate}
\item If the blow-up limit is an ALF or $\ALG^*$ gravitational instanton $(X, g, p)$ with $p\in X$,  the Liouville theorem invoked in the proof is that, any $1$-form $\omega$ that satisfies $\Delta_H \omega = 0$
and $\lim\limits_{d_g(\bx, p)\to \infty}|\omega(\bx)|\to 0$ has to be vanishing everywhere, i.e., $\omega\equiv 0$ on $X$.

\item If the blow-up limit is a flat space in Proposition \ref{p:regularity-scale-I_v*}, namely $\dR^3$,  $\dR^2\times S^1$, $\dR^2$, $\dR^2/\dZ_2$, then we will quote the following Liouville theorem: any harmonic function $f$ that satisfies $|f|\leq C\cdot r^{\mu}$ for any $r\in(0,\infty)$ has to be identically zero, where $r$ is the Euclidean distance to a fixed point.

\item  The Liouville theorem corresponding to $(\PP^1,d_{ML})$ is Proposition 7.8 in \cite{CVZ}.
\end{enumerate}

Combining the above results, now we prove the perturbation theorem.
\begin{theorem}\label{t:perturbation-to-hyperkaehler}
Let $(X, g^{X}, \bm{\omega}^X)$ be an order $2$ $\ALG_{\nu}^*$ gravitational instanton for some $\nu\in\{1,2,3,4\}$. Then for any integer $1 \leq b \leq 14$ and for any sufficiently small parameter $\lambda \ll 1$, there exists a family of hyperk\"ahler structures $(\cM_{\lambda},h_{\lambda}, \bm{\omega}_{h_{\lambda}})$ on the K3 surface $\cM_{\lambda}$ such that
the following properties hold as $\lambda\to0$.

\begin{enumerate}
\item We have Gromov-Hausdorff convergence 
$(\cM_{\lambda},h_{\lambda})\xrightarrow{GH} (\PP^1, d_{ML})$, where $d_{ML}$ is the McLean metric on $\PP^1$  with a finite singular set
$\cS \equiv\{q_0, q_1, \ldots, q_{18-b}\}\subset \PP^1$. Moreover, the curvatures of $h_{\lambda}$ are uniformly bounded away from $\cS$, but are unbounded around $\cS$.

\item The hyperk\"ahler structures $(\cM_{\lambda}, h_{\lambda}, \bm{\omega}_{h_{\lambda}})$ satisfy the uniform error estimate for some positive number $0<\epsilon\ll\min\{1,\frac{\nu}{b}\}$,
\begin{align}
\|h_{\lambda} - g_{\lambda} \|_{C_{0}^{0,\alpha}(\cM_{\lambda})} &\leq C\cdot (\lambda^{2-\epsilon} + \lambda^{\frac{\nu}{b}-\epsilon}), \\
\|\bm{\omega}_{h_{\lambda}} - \bm{\omega}_{\lambda} \|_{C_{0}^{0,\alpha}(\cM_{\lambda})} &\leq  C\cdot (\lambda^{2-\epsilon} + \lambda^{\frac{\nu}{b}-\epsilon}), \label{e:error-deviation}\end{align}
where $g_{\lambda}$ is the metric determined by the definite triple $\bm{\omega}_{\lambda}$, and $C_{0}^{0,\alpha}$-norm is the weighted norm in Definition~\ref{d:weighted-space} when $k=0$ and $\mu=0$.

\item  Rescalings of $(\cM_{\lambda}, h_{\lambda}, \bm{\omega}_{h_{\lambda}})$ around $q_i$ for $1 \leq i\leq 18 - b$ converge to a complete Taub-NUT gravitational instanton on $\dC^2$.

\item   Rescalings of $(\cM_{\lambda}, h_{\lambda}, \bm{\omega}_{h_{\lambda}})$ around $q_0$  converge to the given $\ALG_{\nu}^*$ gravitational instanton $(X, g^{X}, \bm{\omega}^X)$ or one of $(\nu+b)$ copies of complete Taub-NUT gravitational instantons.
 \end{enumerate}
\end{theorem}

\begin{proof}
[Proof of Theorem \ref{t:perturbation-to-hyperkaehler}]
We will apply Lemma \ref{l:implicit-function} to perform the perturbation. Let
\begin{align}
C_{\er} &\equiv 	C_0\cdot (\lambda^2 + \tilde{\lambda} \cdot \ft^3) \cdot \tilde{\lambda}^{- \mu + 1}\cdot \ft^{- \mu - 1} \cdot V(\sigma^{-1})^{\frac{-\mu-1}{2}} + C_0 \cdot \lambda^2 \cdot \tilde{\lambda}^2, \\ 
C_N &\equiv  (\tilde{\lambda}\cdot \lambda)^{\mu - 1} \cdot (T^{\flat})^{\frac{1 - \mu}{2}},\label{e:C-err}
\end{align}
be the constants in Corollary \ref{c:weighted-error} and Lemma \ref{l:nonlinear-term}.
Recall that
$\lambda$ and $\ft$ are chosen such that $\sigma=\frac{\lambda}{\ft} \to 0$. To prove \eqref{e:error-deviation}, we only need to fix the parameter $\ft\equiv \lambda^{\frac{\epsilon}{10}}$ for a fixed $\epsilon\ll1$ and let $\mu = - 1 + \frac{\epsilon}{10}$. Then
it is obvious that $C_{\er} \cdot C_N \to 0$ as $\lambda \to 0$. The uniform linear estimate is given by Proposition \ref{p:injectivity-for-L}.
Then Lemma \ref{l:implicit-function} implies that there exists a solution which satisfies the desired estimate. Moreover, \eqref{e:error-deviation} follows from \eqref{e:C-err}.
 The classification of the intermediate bubbles is given by Proposition \ref{p:regularity-scale-I_v*} and noticing that the solutions $h_{\lambda}$ are sufficiently close to $g_{\lambda}$.
\end{proof}

\section{Proofs of Torelli uniqueness theorems}
\label{s:Torelli}
In this section, we complete the proofs of Theorem~\ref{t:ALGstar2} 
and Theorem~\ref{t:ALG-uniqueness}. We also explain the reason for the order 2 assumption in Theorem~\ref{t:ALG-uniqueness}.

\subsection{Proof of Theorem \ref{t:ALGstar2}: ALG$^*$ Torelli uniqueness}
Let $(X_{\nu}, g, \bm{\omega})$ and $(X_{\nu},g', \bm{\omega}')$ be ALG$^*$ gravitational instantons on $X_{\nu}$ with the same parameters $\kappa_0, L$, which are both order $2$ with respect to the coordinates $\Phi_{X_{\nu}}$ and which satisfy \eqref{e:wdr}.
Let $\pi_{\cK}: \cK \rightarrow \PP^1$ be any elliptic $\K3$ surface with a single fiber of type $\I_b^*$, call it $D^*$, but has all other singular fibers of type~$\I_1$.
Let $U = \{\bx \in \cM_\lambda|\tilde{r}(\bx)\geq \mathfrak{t}\}$ and $V = \{\bx \in \cM_\lambda|\tilde{r}(\bx)\leq 2 \mathfrak{t}\}$. 
Then $\cM_\lambda = U \cup V$. The gluing procedure in Section~\ref{ss:attach} produces approximate hyperk\"ahler triples $ \tilde{ \bm{\omega}}_{\lambda}$ 
and  $\tilde{ \bm{\omega}}_{\lambda}'$ on $\cM_\lambda$.  Note that $U \cap V$ deformation retracts onto the $3$-manifold $\mathcal{I}_{\nu}^3$.
\begin{lemma} The manifold $\mathcal{I}_{\nu}^3 = \Nil^3_{2 \nu}/ \ZZ_2$ is an infranilmanifold, which is a circle bundle of degree $\nu$ over a Klein bottle. Furthermore, we have
$b^1(\mathcal{I}_{\nu}^3) = 1$, with $H^1_{\mathrm{dR}}(\mathcal{I}_{\nu}^3)$ generated by the $1$-form $d \theta_1$.
\label{l:betti-infranil}
\end{lemma}
\begin{proof}
The first statement follows since $\Nil^3_{2\nu}$ is a circle bundle over a torus, and the quotient space is then clearly a circle bundle over a Klein bottle. 
From \cite[Proposition~2.3]{HSVZ}, we have $b^1(\Nil^3_{2\nu}) = 2$, with $H^1_{\mathrm{dR}}(\Nil^3_{2\nu})$ generated by $d \theta_1$ and $d \theta_2$. These forms are harmonic with respect to any left-invariant and $\ZZ_2$-invariant metric on $\Nil^3_{2\nu}$. Of these generators, only $d \theta_1$ is invariant under this action, so the lemma follows from the Hodge Theorem.  
\end{proof}
The Mayer-Vietoris sequence in de Rham cohomology for $\{U, V\}$ is
\begin{equation}
\label{MVK3}
\begin{tikzcd}
 0 \arrow[r] &   H^{1}_{\mathrm{dR}}(U) \oplus H^{1}_{\mathrm{dR}}(V) \arrow[r] \arrow[d, phantom, ""{coordinate, name=Z}] & H^{1}_{\mathrm{dR}}(\mathcal{I}_{\nu}^3)
 \arrow[dll,
rounded corners,
to path={ -- ([xshift=2ex]\tikztostart.east)
|- (Z) [near end]\tikztonodes
-| ([xshift=-2ex]\tikztotarget.west)
-- (\tikztotarget)}]   & \\
 H^{2}_{\mathrm{dR}}(\cM_\lambda) \arrow[r] &    H^{2}_{\mathrm{dR}}(U) \oplus H^{2}_{\mathrm{dR}}(V)  \arrow[r]  &  H^{2}_{\mathrm{dR}}(\mathcal{I}_{\nu}^3) \arrow[r]  & 0 .
\end{tikzcd}
\end{equation}
From the gluing in Section \ref{s:gluing-construction}, we have
\begin{align}
(\Phi^{-1})^*(\lambda^2 \cdot \bm{\omega}) - \bm{\omega}^{\fN} = d \bm{\eta}, \quad
(\Phi^{-1})^*(\lambda^2 \cdot \bm{\omega}') - \bm{\omega}^{\fN} = d \bm{\eta}'.
\end{align} where $\bm{\eta}$ and  $\bm{\eta}'$ are triples of $1$-forms on $\{\bx\in \fN|\ft\leq \tilde{r}(\bx)\leq 2\ft\}$.
From Lemma~\ref{l:approximate-triple}, on the region $\{\bx \in \fN|\mathfrak{t} \leq \tilde{r}(\bx)\leq 2 \mathfrak{t}\}$, the approximate hyperk\"ahler triples are
\begin{align}
\tilde{\bm{\omega}} =  
\bm{\omega}^{\fN} + d (\varphi\cdot \bm{\eta}), \quad \tilde{\bm{\omega}}' =  \bm{\omega}^{\fN} + d (\varphi\cdot \bm{\eta}'),
\end{align}
where $\varphi$ is a cut-off function which is 1 when $\tilde{r}(\bx) \leq \mathfrak{t}$, and is 0 when $\tilde{r}(\bx) \geq 2\mathfrak{t}$.
Clearly, the image of $[ \tilde{\omega}_i ] \in H^{2}_{\mathrm{dR}}(\cM_\lambda)$ in $H^{2}_{\mathrm{dR}}(U) \oplus H^{2}_{\mathrm{dR}}(V)$ is $([\omega_i], [\omega_i^{V}])$. Since the two $\ALG^*$  gravitational instantons have the same $[\omega_i] = [\omega_i']$ and we also use the same $\omega_i^{V}$ for both, we see that the image of $[\omega_i]$ and $[\omega_i']$ are the same. So their difference is in the image of $H^{1}_{\mathrm{dR}}(\mathcal{I}_{\nu}^3)$.
To see the image, we start with $d \theta_1 \in H^{1}_{\mathrm{dR}}(\mathcal{I}_{\nu}^3)$. It can be written as the difference of $\varphi d\theta_1$ on $U$ and $(\varphi-1) d \theta_1$ on $V$. The form 
$d (\varphi d\theta_1) = d((\varphi-1) d \theta_1)
$ can be viewed as a two form on $\cM_\lambda$ which is the image of $d \theta_1$ in $H^{2}_{\mathrm{dR}}(\cM_\lambda)$. Therefore,$[\tilde{\omega}_i]$ and $[\tilde{\omega}_i']$ may differ by a multiple of $[d (\varphi d\theta_1)]$. Fortunately, we can modify the 1-form $\eta_i$ by the same multiple of $d\theta_1$, and we then obtain $[\tilde{\omega}_i] = [\tilde{\omega}_i']
\in H^2_{\mathrm{dR}}(\cM_\lambda)$. Remark that this modification will not affect any of the estimates in the proof of the gluing theorem. 

Then we need to perturb the approximate hyperk\"ahler triples to be actually hyperk\"ahler. The resulting cohomology classes will not be exactly the same any more, but the span of them will remain the same. Therefore, by a rescaling and a hyperk\"ahler rotation, we can get the same $[\omega_i^{HK}]$ on $\cM_\lambda$. Observe that the rescaling factor converges to 1 and the hyperk\"ahler rotation matrix converges to the identity matrix as $\lambda \to 0$. By the Torelli-type theorem for K3 surfaces, there exists an isometry between them which maps the hyperk\"ahler triples onto each other, and induces the identity mapping on $H^2(\cM_\lambda)$;
see \cite{Besse, BurnsRapoport, ShapiroShafarevitch}.  Therefore, the restriction of these maps to the $\ALG^*$ bubbling regions will then converge to an isometry of the $\ALG^*$ spaces as $\lambda \to 0$, since the isometry must map the $\ALG^*$ regions to each other. Obviously, this isometry will map the hyperk\"ahler triples onto each other.
The homology class of a fiber generates $H_2(\mathcal{I}_{\nu}^3;\RR) = \RR$ and is nontrivial in both $H_2(U;\RR)$ and $H_2(V;\RR)$ under the natural inclusions. From the Mayer-Vietoris sequence in homology, it follows the natural mapping $H_2(V;\RR) \to H_2 (\cM_\lambda;\RR)$ is injective. By duality, the restriction $H^2(\cM_\lambda;\RR) \to H^2 (V;\RR) \cong H^2 (X_{\nu};\RR)$ is surjective, which implies that the isometry of the $\ALG^*$ regions also induces the identity map on $H^2(X_{\nu};\RR) \cong H^2_{\mathrm{dR}}(X_{\nu})$, so we are done. 

\subsection{Proof of Theorem \ref{t:ALG-uniqueness}: ALG Torelli uniqueness}
\label{s:ALG-Torelli}

 The next goal is to prove Theorem \ref{t:ALG-uniqueness} which requires the following gluing result.  
\begin{theorem}\label{t:ALG-gluing}	
Let $(X, g^{X}, \bm{\omega}^X)$ be an $\ALG$ gravitational instanton of order $2$ with $\chi(X)=\chi_0$. Then there exists a family of hyperk\"ahler structures $(\cM_{\lambda}, h_{\lambda},\bm{\omega}_{h_{\lambda}})$ on the K3 surface $\cM_{\lambda}$ such that
the following holds as $\lambda\to 0$.

\begin{enumerate}
\item We have Gromov-Hausdorff convergence 
$(\cM_{\lambda},h_{\lambda})\xrightarrow{GH} (\PP^1, d_{ML})$, where $d_{ML}$ is the McLean metric on $\PP^1$  with a finite singular set
$\cS\equiv\{q_0, q_1, \ldots, q_{24-\chi_0}\}\subset \PP^1$. Moreover, the curvatures of $h_{\lambda}$ are uniformly bounded away from $\cS$, but are unbounded around $\cS$.

\item Rescalings of $(\cM_{\lambda}, h_{\lambda}, \bm{\omega}_{h_{\lambda}})$ around $q_i$ for $1\leq i\leq 24-\chi_0$ converge to a complete Taub-NUT gravitational instanton on $\dC^2$.

\item Rescalings of $(\cM_{\lambda}, h_{\lambda}, \bm{\omega}_{h_{\lambda}})$ around $q_0$ converge to the given $\ALG$ gravitational instanton $(X, g^{X}, \bm{\omega}^X)$.

\end{enumerate}
\end{theorem}

\begin{proof}The proof is a straightforward generalization of \cite[Theorem~1.1]{CVZ}  using a general hyperk\"ahler triple gluing argument as in Section \ref{s:perturbation}. 
In \cite{CVZ}, we assumed the $\ALG$ gravitational instantons were isotrivial which was necessary to preserve the complex structure. Since we are not fixing the complex structure on the K3 surface, only the order $2$ assumption is necessary. 
For this, we just need to note that \cite[Proposition~5.6]{CVZ} holds for any order $2$ $\ALG$ space; the isotrivial condition is not necessary. \end{proof}

Let $(X_{\beta}, g, \bm{\omega})$ and $(X_{\beta},g', \bm{\omega}')$ be ALG gravitational instantons on $X_{\beta}$ with the same parameters $\beta$, $\tau$, and $L$, which are both order $2$ with respect to the coordinates $\Phi_{X_{\beta}}$ and which satisfy \eqref{e:wdr2}.  The parameter $\beta$ determines a fiber $D$ of type $\I_0^*$, $\II$, $\III$, $\IV$, $\II^*$, $\III^*$, $\IV^*$ as in Table ~\ref{ALGtable}.
Let $\pi_{\cK}: \cK \rightarrow \PP^1$ be any elliptic $\K3$ surface with a single fiber $D^*$ of the dual type, which means $\I_0^*$, $\II^*$, $\III^*$, $\IV^*$, $\II$, $\III$, $\IV$, respectively, but has all other singular fibers of type $\I_1$. We use an attaching map $\Psi$ from $\{\lambda^{-1} \leq r \leq 2 \lambda^{-1}\}\subset X_{\beta}$ to a small annular region in $\cK$ centered around $D^*$ to obtain a manifold $\cM_\lambda$, where $\lambda$ is sufficiently small. Let $U$ be the subset such that $r \ge \lambda^{-1}$ and $V$ be the subset such that $r \leq 2\lambda^{-1}$. Then $\cM_{\lambda} = U \cup V$.

The gluing procedure in the proof of Theorem~\ref{t:ALG-gluing} produces 
approximate hyperk\"ahler triples $ \tilde{ \bm{\omega}}_{\lambda}$ 
and  $\tilde{ \bm{\omega}}_{\lambda}'$ on the $\cM_{\lambda}$. Note that $U \cap V$ deformation retracts onto the $3$-manifold $N_{\beta}^3$ which is the restriction of an elliptic fibration with a single fiber of type $D^*$ to $S^1$.
\begin{lemma}
\label{h1lem}
The manifold $N_{\beta}^3$ is flat, and satisfies $b^1(N_{\beta}^3) = b^2(N_{\beta}^3) = 1$.
Furthermore, a generator for $b^1(N_{\beta}^3)$ is the $1$-form $d \theta_1$, where $\theta_1$ is the angular coordinate on the cone $\mathcal{C}(2\pi \beta)$.
\end{lemma}
\begin{proof}The $3$-manifold
$N_{\beta}^3$ is a $T^2$-fibration over $S^1$. We cover $S^1=\RR/2\pi\beta\ZZ$ by two intervals $(0,2\pi\beta)$ and $(\pi\beta,3\pi\beta)$. Then we can write $N_{\beta}^3$ as the union of 
$N_{\beta,1}^3 \equiv (0,2 \pi \beta)\times T^2$
 and  $N_{\beta,2}^3 \equiv (\pi \beta, 3 \pi \beta)\times T^2$.
The Mayer-Vietoris sequence is
\begin{equation*}
\begin{tikzcd}
H^{0}_{\mathrm{dR}}(N_{\beta}^3)  \arrow[r] &   H^{0}_{\mathrm{dR}}(N_{\beta,1}^3) \oplus H^{0}_{\mathrm{dR}}(N_{\beta,2}^3) \arrow[r] \arrow[d, phantom, ""{coordinate, name=Z}] & H^{0}_{\mathrm{dR}}(N_{\beta,1}^3 \cap N_{\beta,2}^3)
 \arrow[dll,
rounded corners,
to path={ -- ([xshift=2ex]\tikztostart.east)
|- (Z) [near end]\tikztonodes
-| ([xshift=-2ex]\tikztotarget.west)
-- (\tikztotarget)}]   & \\
 H^{1}_{\mathrm{dR}}(N_{\beta}^3) \arrow[r] &    H^{1}_{\mathrm{dR}}(N_{\beta,1}^3) \oplus H^{1}_{\mathrm{dR}}(N_{\beta,2}^3)  \arrow[r]  &  H^{1}_{\mathrm{dR}}(N_{\beta,1}^3 \cap N_{\beta,2}^3).
\end{tikzcd}
\end{equation*}
If the monodromy group is $A$, then the map 
\begin{equation}
H^{1}_{\mathrm{dR}}(N_{\beta,1}^3) \oplus H^{1}_{\mathrm{dR}}(N_{\beta,2}^3) = \RR^2 \oplus \RR^2 \to H^{1}_{\mathrm{dR}}(N_{\beta,1}^3 \cap N_{\beta,2}^3)=\RR^2 \oplus \RR^2
\end{equation}
is given by
$(C_1, C_2) \mapsto (C_1 - C_2, C_1 - A C_2)
$ for $C_1, C_2 \in \RR^2$, whose kernel is the same as $\ker(A-\Id)$. For singular fibers of finite monodromy, $\ker(A-\Id)=0$. The map 
\begin{equation}
H^{0}_{\mathrm{dR}}(N_{\beta,1}^3) \oplus H^{0}_{\mathrm{dR}}(N_{\beta,2}^3)  \to H^{0}_{\mathrm{dR}}(N_{\beta,1}^3 \cap N_{\beta,2}^3)
\end{equation}
is a rank 1 map 
$(a,b)\mapsto (a-b,a-b)$.
 So $H^{1}_{\mathrm{dR}}(N_{\beta}^3)=\RR$ and it is generated by the image of $(2\pi\beta,0)\in H^{0}_{\mathrm{dR}}(N_{\beta,1}^3 \cap N_{\beta,2}^3)$. To see this image, we note that the difference of the function $\theta_1$ on $(\pi\beta,3\pi\beta)\times T^2$ with the function $\theta_1$ on $(0,2\pi\beta)\times T^2$ is exactly $2\pi\beta$ on $(0,\pi\beta)\times T^2$ and $0$ on $(\pi\beta,2\pi\beta)\times T^2$. The derivatives of them are all $d\theta_1$. So the image of $(2\pi\beta,0)$ is $d\theta_1$. In other words, we have proved that $H^{1}_{\mathrm{dR}}(N_{\beta}^3)=\langle d\theta_1 \rangle$. By Poincar\'e duality, $b^2(N_{\beta}^3)=b^1(N_{\beta}^3)=1$. The flatness of $N_{\beta}^3$ is a corollary of the fact that the flat metric on $N_{\beta,1}^3$ and $N_{\beta,2}^3$ can be glued into a flat metric on $N_{\beta}^3$.
\end{proof} 
The proof of Theorem \ref{t:ALG-gluing} uses \cite[Proposition~5.6]{CVZ}, which implies 
\begin{align}
\Phi_{X_{\beta}}^*( \bm{\omega}) - \bm{\omega}^{\cC} = d \bm{\eta}, \quad
\Phi_{X_{\beta}}^*( \bm{\omega}') - \bm{\omega}^{\cC} = d \bm{\eta}'.
\end{align}
for some triples of $1$-forms $\bm{\eta}$ and $\bm{\eta}'$ defined on end of the model space. On the region $U$, away from the damage zone, the approximate hyperk\"ahler triples are exactly the same (they are semi-flat, with $\I_1$ fibers resolved using Ooguri-Vafa metrics). Using the same Mayer-Vietoris sequence \eqref{MVK3}, and Lemma \ref{h1lem},
we can adjust the $1$-form $\eta_i$ on the ``damage zone'' by a term of the form $d( \varphi \theta_1)$, to arrange that $[\omega_i] = [\omega_i']$ in $H^2_{\mathrm{dR}}(\cM_\lambda)$. 
The remainder of the proof is then exactly the same as in the $\ALG^*$ case above.

\section{Results on the period mapping}
\label{s:period}
In this section we will use the following notation.  
In the $\ALG^*$ case, extend $\fs$ to a smooth function $s$ defined on all of $X$ satisfying $s \geq 1$. In the $\ALG$ case, similarly extend $r$ to a smooth function defined on all of $X$, and again denoted the extended function by $s$. 
\subsection{Harmonic 2-forms of order 2}
In order to properly define the period map, we begin with a proposition relating compactly supported de Rham cohomology and decaying harmonic $2$-forms.
\begin{proposition}\label{t:L2-harmonic}
For any $\ALG$ or $\ALG^*$ gravitational instanton $(X,g,\bm{\omega})$ of order~$2$,
\begin{align*}
\{\omega &= O(s^{-2}) \in \Omega^2(X) \ | \ \Delta \omega =0\} = \{\omega = O(s^{-2}) \in \Omega^2(X) \ | \  d \omega = d^*\omega = 0\}\\
& = \{\omega = O(s^{-2}) \in \Omega^2_-(X) \ | \ d \omega = d^*\omega = 0\}\\
&= \Ima (H^2_{\mathrm{cpt}}(X)\to H^2(X)) = \Big\{[\omega]\in H^2(X), \int_D \omega = 0\Big\},
\end{align*}
where $D$ is any fiber arising from the compactification of $X$ to a rational elliptic surface.
\end{proposition}
\begin{proof}
We first consider the $\ALG^*$ case.
If $\omega = O(s^{-2}) \in \Omega^2(X)$, and $\Delta \omega =0$, by standard elliptic regularity, $\omega\in W^{k,2}_{\mu}$ for any $k \in \dN_0$ and $\mu>-2$. So the boundary term in
\begin{equation}
\int_{r<R} \Big( (\omega, \Delta \omega) - (d\omega, d\omega) - (d^*\omega, d^*\omega) \Big)
\end{equation}
goes to 0 when $R\to\infty$, which implies $d\omega = d^*\omega=0$. Conversely, if $d\omega = d^*\omega =0$, then $\Delta\omega=0$.

Then we study $\Ima (H^2_{\mathrm{cpt}}(X)\to H^2(X))$. Define $U = \{x \in X, r(x) > R\}$, then $U$ deformation retracts to the 3-manifold $\mathcal{I}_{\nu}^3$. By Lemma \ref{l:betti-infranil}, $H^1(\mathcal{I}_{\nu}^3)$ is generated by $d\t1$. By Poincar\'e duality, $H_2(\mathcal{I}_{\nu}^3)$ is generated by $[D]$, where $D$ is any fiber, so $H_2(U)$ is also generated by $[D]$. Therefore, if $[\omega]\in H^2(X)$ and $\int_D \omega = 0$ then $\omega|_U$ is exact, so there exists $\eta \in \Omega^1(U)$ such that $\omega = d\eta$ on $U$. Let $\chi$ be a cut-off function which is 0 when $r\le R$ and is 1 when $r\ge 2R$. Then $\omega - d(\chi \cdot \eta)$ is compactly supported, so $[\omega] \in \Ima (H^2_{\mathrm{cpt}}(X)\to H^2(X))$. Conversely, if $\omega$ is compactly supported, then it is trivial to see that $\int_D \omega=0$.

For any $\omega = O(s^{-2}) \in \Omega^2(X)$ such that $\Delta\omega=0$, it is easy to see that $\int_D \omega=0$ by choosing far enough $D$. So there is a map
\begin{align*}
\{\omega = O(s^{-2}) \in \Omega^2(X): d \omega = d^*\omega = 0\} \to \left\{[\omega]\in H^2(X): \int_D \omega = 0\right\}.
\end{align*}
To show the surjectivity, for any compactly supported closed form $\omega$, choose an arbitrary $0<\epsilon<1$ and a basis $\eta_i$ of 2-forms in $W^{k,2}_{-1-\epsilon}$ such that $\Delta\eta_i=0$. Since $(\eta_i, \eta_j)_{L^2}$ is invertible, there exist $c_i \in \RR$ such that 
\begin{equation}
(\omega, \eta_j)_{L^2} = (\sum_i c_i \eta_i, \eta_j)_{L^2}.
\end{equation}
By \cite[Proposition~4.2(2)]{CVZ2}, there exists $\phi \in W^{k+2, 2}_{1+\epsilon}$ such that 
\begin{equation}
\Delta\phi =\omega - \sum_i c_i \eta_i.
\end{equation}
Since 
\begin{equation}
\int_{r<R} \Big( (\eta_i, \Delta \eta_i) - (d\eta_i, d\eta_i) - (d^*\eta_i, d^*\eta_i) \Big)
\end{equation}
also decays as $R \to \infty$, $\eta_i$ are closed and co-closed. So,
\begin{equation}
\omega - d d^*\phi = d^*d \phi + \sum_i c_i \eta_i \in W^{k,2}_{-1+\epsilon}
\end{equation}
is closed and co-closed. The self-dual part is $\sum_{i=1}^{3} f_i \omega_i$ for decaying harmonic functions $f_i$, which must be 0. By (A.72), (A.82), (A.120), (A.121) and (A.136)-(A.142)  of \cite{CVZ2}, the closed and co-closed anti-self-dual form $\omega - d d^*\phi$ must be $O(s^{-2})$, which implies the surjectivity.

To show the injectivity, assume that $d\phi = \omega = O(s^{-2}) \in \Omega^2(X)$ is also co-closed. We write $\phi = d r \wedge \alpha + \beta$, where $\alpha$ is a 1-form on $\mathcal{I}_{\nu}^3=\{r=r_0\}$, and $\beta$  is a 2-form on $\{r=r_0\}$. Then 
\begin{equation}
0 = d \phi = - d r \wedge d_{\mathcal{I}_{\nu}^3}\alpha + d_{\mathcal{I}_{\nu}^3}\beta + d r \wedge \frac{\partial \beta}{\partial r}.
\end{equation}
Define $\gamma = \int_{\infty}^{r} \alpha$ on $U$, then
\begin{equation}
d \gamma = d r \wedge \alpha + \int_{\infty}^{r} d_{\mathcal{I}_{\nu}^3} \alpha = d r \wedge \alpha + \int_{\infty}^{r} \frac{\partial \beta}{\partial r} = d r \wedge \alpha + \beta = \omega.
\end{equation}
So, $\gamma-\phi$ is closed on $U$. By Lemma \ref{l:betti-infranil}, $H^1(U)$ is generated by $d\t1$. So there exist a constant $c$ and a function $\psi$ on $U$ such that $\gamma-\phi= c\t1+d\psi$. Then $\omega = d \eta$, where
$\eta = \phi + d (\chi\cdot \psi)$. Moreover, $\eta = \gamma -c\t1$ when $r\ge 2R$. So $\eta\in W^{k+1,2}_{-1+\epsilon}$ for any $\epsilon>0$. Therefore,
\begin{equation}
\int_{r<R} \Big( (\omega, d\eta) - (d^*\omega, \eta) \Big)
\end{equation}
also converges to $0$ as $R \to \infty$. In other words, $\omega=d\eta=0$ since $d^*\omega=0$.

Using the same proof, and the ALG asymptotic analysis in \cite{CCI}, a similar proof also holds for ALG gravitational instantons of order 2. See also \cite[Section 7.1.3]{HHM} and Theorems 9.3 and 9.4 of \cite{CVZ}.

\end{proof}

\subsection{Definition of the period map}

In this subsection, we prove the period mappings are well-defined. 
\begin{proposition}  The period mappings $\mathscr{P}$ in Definition~\ref{d:ALGperiod} and Definition~\ref{d:ALGstarperiod} are well-defined. 
\end{proposition}
\begin{proof}We first consider the ALG case. 
If $(X_{\beta}, g, \bm{\omega}) \in  \mathcal{M}_{\beta, \tau, L}$, then it is $\ALG$ with respect to the fixed $\ALG$ coordinate system $\Phi_{X_{\beta}}$. Then $\omega_1$ is taken to be the K\"ahler form which is asymptotic to the elliptic complex structure, and the choice of  $\omega_2$ and $\omega_3$ is also determined since they are aymptotic to the model K\"ahler forms in the  $\Phi_{X_{\beta}}$  coordinates. The point is our Definition \ref{d:ALG-space} removes the freedom of hyperk\"ahler rotations, so we have a well-defined ordered choice of the 3 K\"ahler forms. From \cite[Theorem~4.14]{CCI}, there is a holomorphic function $u : X_{\beta} \rightarrow \CC$ which is an elliptic fibration. The level sets of $u$ are tori. As $u \to \infty$, these level sets are close to the model holomorphic tori. Therefore the homology class $[D]$ of any fiber is well-defined, the same class for all elements in $ \mathcal{M}_{\beta, \tau, L}$.
Since the forms $\omega_2$ and $\omega_3$ are orthogonal to $\omega_1$, any torus which is holomorphic for $I$ is Lagrangian with respect to $J$ or $K$.  Use Proposition \ref{t:L2-harmonic} to identify $\mathscr{H}^2$ with order $2$ decaying harmonic anti-self-dual $2$-forms, the classes $[\omega_2]$ and $[\omega_3]$ automatically lie in $\mathscr{H}^2$. Finally, since the holomorphic tori for $I$ and $I_0$ are homologous, we have 
$\int_D (\omega_1 - \omega^0_1) = 0$ since the area of the holomorphic tori are the same. 
Using  \cite[Proposition~3.1]{CV}, the argument in the $\ALG^*$ case is exactly the same.
 \end{proof}

\subsection{The nondegeneracy condition}
In this subsection, we prove the nondegeneracy condition stated in Theorems \ref{t:ALGperiod} and \ref{t:ALGstarperiod}:
\begin{align}
\label{nondegcond}
\bm{\omega}[C] 
\not=(0,0,0) \mbox{ for all } [C] \in H_2(X; \ZZ)
\mbox{ satisfying }   [C]^2=-2.
\end{align}
To prove this,  we use the gluing construction in Theorem~\ref{t:ALG-gluing} in the ALG case and Theorem~\ref{t:perturbation-to-hyperkaehler} in the ALG$^*$ case. 
A basic tranversality argument shows that we can represent any $[C]  \in H_2(X; \ZZ)$ by an embedded surface $\iota : C \rightarrow X$. If \eqref{nondegcond} is not satisfied by an ALG or ALG$^*$ gravitational instanton $(X,\bm{\omega}^X)$, then by choosing a small enough gluing parameter $\lambda$, we can assume that the glued closed definite triple $\bm{\omega}_\lambda=\bm{\omega}^X$ near $\iota(C)$. A Mayer-Vietoris argument in homology shows that $[C]$ is nontrivial in  $H_2(\mathcal{M}_{\lambda}, \ZZ)$.  So, there exists $[C]\in H_2(\mathcal{M}_{\lambda}, \ZZ)$ such that $[C]^2=-2$ and $[\bm{\omega}_\lambda] \cdot [C] = \bm{0}$. In the perturbation arguments, the span of the hyperk\"ahler classes $[\bm{\omega}^{HK}_\lambda]$ on the K3 surface $\mathcal{M}_{\lambda}$ is the same as the span of $[\bm{\omega}_\lambda]$. Therefore, $[\bm{\omega}^{HK}_\lambda] \cdot [C] = \bm{0}$, which is a contradiction with the well-known nondegeneracy condition on the K3 surface $\mathcal{M}_{\lambda}$.
 
\subsection{Proofs of  Theorem~\ref{t:ALGperiod}
and Theorem~\ref{t:ALGstarperiod}.}
We follow the route map of \cite[Section~7]{CCIII}. For any point in $\mathscr{H}^2 \oplus\mathscr{H}^2 \oplus \mathscr{H}^2$ whose sum with $\bm{\omega}^0$ satisfies \eqref{nondegcond}, we can connect it to $(0,0,0)$ by zigzags of the form 
\begin{equation}
([\alpha_{1,0}]+t[\beta_1], [\alpha_2], [\alpha_3]),\label{CaseI}
\end{equation}
\begin{equation}
([\alpha_1], [\alpha_{2,0}]+t[\beta_2], [\alpha_3]),\label{CaseJ}
\end{equation}
or
\begin{equation}
([\alpha_1],[\alpha_2],[\alpha_{3,0}]+t[\beta_3]).\label{CaseK}
\end{equation}
We require that the sum of $\bm{\omega}^0$ with all the points in the zigzags satisfy \eqref{nondegcond}.
Let us consider the ALG case. For the path in \eqref{CaseI} we have, 
\begin{equation}
(X,\omega_{1,0}\equiv\omega_1^0+\alpha_{1,0},\omega_2\equiv\omega_2^0+\alpha_2,\omega_3\equiv\omega_3^0+\alpha_3)\in \mathcal{M}_{\beta, \tau, L}.
\end{equation}

Using Proposition \ref{t:L2-harmonic}, we choose the representative $\beta_1$ in the class $[\beta_1]$ by requiring it to be closed, co-closed, and anti-self-dual with respect to the hyperk\"ahler metric determined by $(X,\omega_{1,0},\omega_2,\omega_3)$. 
Since $\beta$ is anti-self-dual,
\begin{align}
\label{betaasd}
\beta_1\wedge\omega_{1,0}=\beta_1\wedge\omega_2=\beta_1\wedge\omega_3=0.
\end{align}
Then we choose $s_t\in\RR$ such that 
\begin{equation}
\omega_{1,t}\equiv \omega_{1,0}+t\beta_1+s_t\sqrt{-1}\partial_I\bar\partial_I (\chi \cdot \log|u|)
\end{equation}
satisfies
\begin{equation}
\int_X (\omega_{1,t}^2-\omega_{1,0}^2)=0,
\end{equation}
where $I,J,K$ are the hyperk\"ahler structures determined by $(X,\omega_{1,0},\omega_2,\omega_3)$, $u : X \to \CC$ is the $I$-holomorphic function which makes $X$ a rational elliptic surface minus the fiber at infinity, and $\chi$ is a cut-off function which is 0 for small $|u|$ and is 1 for large $|u|$. The constant $s_t$ exists because 
\begin{equation}
\int_X (\omega_{1,t}^2-\omega_{1,0}^2) = \int_X t^2\beta_1^2 + 2s_t\int_X \omega_{1,0}\wedge \sqrt{-1}\partial_I\bar\partial_I (\chi \cdot \log|u|),
\end{equation}
and $\int_X \omega_{1,0}\wedge \sqrt{-1}\partial_I\bar\partial_I (\chi \cdot \log|u|)\not=0$.
In the \eqref{CaseJ} case, we use $\sqrt{-1}\partial_J\bar\partial_J$ instead of $\sqrt{-1}\partial_I\bar\partial_I$, and in the \eqref{CaseK} case, we use $\sqrt{-1}\partial_K\bar\partial_K$.
\begin{lemma}
\label{l:Sobolev}
Let $(X,g,p)$ be a  gravitational instanton of type  either $\ALG$ or $\ALG^*$ of order $2$, where $p \in X$. For $\delta \in \RR$, there exists a constant $C$ so that
\begin{align}
\sup\limits_{X} | s^{-\delta} \varphi | \leq C \Vert \varphi \Vert_{ W^{100,2}_{\delta}(X)},\label{e:weighted-C_0-Sobolev}
\end{align}
for all $\varphi \in  W^{100,2}_{\delta}(X)$.
\end{lemma}
\begin{proof}The proof is a standard rescaling argument. It suffices to prove the following inequality on any large annulus. We will show that there exists a constant $C = C(\delta) > 0$ such that for any sufficiently large constant $\zeta\gg1$, 
\begin{align}
\sup\limits_{A(\zeta, 2\zeta, p)} | s^{-\delta} \varphi | \leq C \sum\limits_{m=0}^{100}\Vert \nabla^m \varphi \Vert_{L^2_{\delta}(A(\zeta, 2\zeta, p))},\quad \forall\  \varphi \in  W^{100,2}_{\delta}(A(\zeta, 2\zeta, p)),
\end{align} where $A(\zeta, 2\zeta, p)\equiv \{\bx\in X | \ \zeta \leq s(\bx)\leq 2\zeta\}$. Let us consider the rescaled metric $\tilde{g} \equiv\zeta^2 \cdot g$ such that $g$ is {\it non-collapsing} in the annulus $A(\zeta, 2\zeta, p)$. Then the desired inequality will follow from the standard Sobolev inequality,
\begin{align}
	\sup\limits_{\widetilde{A}_{\zeta}} | \varphi | \leq C \sum\limits_{m=0}^{100}\frac{1}{\Vol_{\tilde{g}}(\widetilde{A}_{\zeta})}\int_{\widetilde{A}_{\zeta}} |\nabla^m \varphi|^2\dvol_{\tilde{g}},\end{align}
where $C>0$ is independent of $\zeta$ and $\widetilde{A}_{\zeta}$ is the rescaled image of the annulus $A(\zeta,2\zeta, p)$. Finally, one obtains \eqref{e:weighted-C_0-Sobolev} by a simple covering argument.
\end{proof}
Back to the \eqref{CaseI} case, consider the collection $\mathcal{S}$ of $t\in[0,1]$ for which there exists $\delta_t>0$ and $\varphi_t\in W^{k,2}_{-\delta_t}(X, \omega_{1,0})$ for any $k\in\dN$ such that
\begin{equation}
(X,\omega_t\equiv\omega_{1,t}+\sqrt{-1}\partial_I\bar\partial_I\varphi_t,\omega_2,\omega_3)\in \mathcal{M}_{\beta, \tau, L}.
\end{equation}
By assumption, $0 \in \mathcal{S}$. If $t_0 \in \mathcal{S}$, then by definition of $\mathcal{S}$, iterating Lemma~\ref{l:Sobolev} and using a standard elliptic regularity argument, we see that  for $t$ sufficiently close to $t_0$, $\omega_{1,t}+\sqrt{-1}\partial_I\bar\partial_I\varphi_{t_0}$ will be ALG or ALG$^*$ of order $2$. Furthermore,
\begin{equation}
\int_X (\omega_{1,t}+\sqrt{-1}\partial_I\bar\partial_I\varphi_{t_0})^2 - \omega_{1,0}^2 = \int_X (\omega_{1,t}^2-\omega_{1,0}^2) = 0.
\end{equation}
 By \cite[Theorem~1.1]{TianYau}, there exists a bounded solution $\varphi_t$ of the equation 
\begin{align}
( \omega_{1,t} +  \sqrt{-1}\partial_I\bar\partial_I\varphi_t)^2 = e^f \omega_{1,t}^2,
\end{align}
where 
\begin{align} 
f \equiv \log \frac{\omega_{t}^2-\omega_{1,t}^2}{\omega_{1,t}^2} = \log \frac{\omega_{1,0}^2-\omega_{1,t}^2}{\omega_{1,t}^2} = O(r^{-4}),
\end{align} 
and the middle equality follows from \eqref{betaasd}.  
By \cite[Proposition~2.6]{Hein}, $\int_X |\nabla_{\omega_{1,t}} \varphi_t|^2 \omega_{1,t}^2 < \infty$. 
Then, by \cite[Proposition~2.9 (ib)]{Hein}, there exists a $\delta_t > 0$ so that 
\begin{align}
\sup |\varphi_t|  \leq C s^{- \delta_t}. 
\end{align}
Then, \cite[Proposition~2.9 (ii)]{Hein} implies that 
\begin{align}
\sup |\nabla^k_{\omega_{1,t}} \varphi_t| \leq C_k s^{- \delta_t - k},
\end{align}
since these estimates are implied by Hein's weighted H\"older estimates. 
This implies that 
$\varphi_t \in W^{k,2}_{-\delta_t}(X, \omega_{1,0})$ for any $k\in\dN$ if we slightly shrink $\delta_t$. Consequently, $\mathcal{S}$ is open. 
This also implies that the period mapping is an open mapping. The above arguments hold in the ALG$^*$ case (with $\mathcal{M}_{\beta, \tau, L}$ replaced by  $\mathcal{M}_{\nu, \kappa_0, L}$), so this completes the proof of Theorem~\ref{t:ALGperiod}.

To finish the proof of Theorem~\ref{t:ALGstarperiod}, we need to show that $\mathcal{S}$ is closed in the ALG cases. So suppose that $t_i\to t_\infty$ is a sequence in $\mathcal{S}$, then
\begin{equation}
\begin{split}
\int_X (\mathrm{tr}_{\omega_{1,0}}\omega_{t_i}-2)\frac{\omega_{1,0}^2}{2}=\int_X \omega_{1,0}\wedge(\omega_{t_i}-\omega_{1,0})=\int_X \omega_{1,0}\wedge(\omega_{1,t_i}-\omega_{1,0})\\
=s_{t_i}\int_X \omega_{1,0}\wedge \sqrt{-1}\partial_I\bar\partial_I (\chi \cdot \log|u|) = -\frac{t_i^2}{2} \int_X \beta_1^2 \le C
\end{split}
\label{e:energy-bound}
\end{equation}
for a constant independent of $t_i$, and
\begin{equation}
\begin{split}
&\int_X (\mathrm{tr}_{\omega_{t_j}}\omega_{t_i}-2)\frac{\omega_{t_j}^2}{2}=\int_X \omega_{t_j}\wedge(\omega_{t_i}-\omega_{t_j})=\int_X \omega_{1,t_j}\wedge(\omega_{1,t_i}-\omega_{1,t_j})\\
&=\int_X (\omega_{1,0}+t_j\beta_1)\wedge((t_i-t_j)\beta_1+(s_{t_i}-s_{t_j})\sqrt{-1}\partial_I\bar\partial_I (\chi \cdot \log|u|))\\
&=(s_{t_i}-s_{t_j})\int_X \omega_{1,0}\wedge \sqrt{-1}\partial_I\bar\partial_I (\chi \cdot \log|u|)+t_j(t_i-t_j)\int_X \beta_1^2\\
&=(-\frac{t_i^2}{2}+\frac{t_j^2}{2}+t_j(t_i-t_j))\int_X \beta_1^2 = -\frac{(t_i-t_j)^2}{2} \int_X \beta_1^2 \to 0
\end{split}
\label{e:energy-converge}
\end{equation}
as $i,j\to \infty$. These bounds imply the following pointwise bound. 
\begin{theorem} The function $e(t_i) = \mathrm{tr}_{\omega_0} \omega_{t_i} = \mathrm{tr}_{\omega_{t_i}} \omega_0$ is uniformly bounded on $X$. 
\end{theorem}
\begin{proof}
We use \eqref{e:energy-bound} and \eqref{e:energy-converge} to go through the arguments in \cite[Section~7]{CCIII}, with some minor modifications, to get the required bound. First, cover $X$ by balls with radius~1 in the sense of the metric determined by $(X,\omega_{1,0},\omega_2,\omega_3)$ such that the number of balls containing any point in $X$ is uniformly bounded. Then we use these balls to replace the sets $U_N$ in \cite[Theorem~7.3]{CCIII}, which yields a bound on the diameter (in the metric $\omega_{t_i}$) of these balls. Note that the proof of \cite[Lemma~1.3]{DPS} is valid in the ALG case, since ALG metrics are volume non-collapsed in bounded scales at infinity. 

To prove the analogue of \cite[Theorem~7.4]{CCIII}, we need to show that if there exists a sequence of cohomology classes $[\Sigma_i]\in H^2(X,\ZZ)$ satisfying $[\Sigma_i]^2=-2$ and $\int_{\Sigma_i} \omega_{t_i}\to 0, \int_{\Sigma_i} \omega_2 \to 0, \int_{\Sigma_i} \omega_3 \to 0$ as $i\to\infty$, then there are only finite many distinct $[\Sigma_i]$. 
To prove this, recall that by the assumption of Theorem \ref{t:ALGperiod} based on \cite[Theorem~1.10]{CV}, $X$ is in particular diffeomorphic to an isotrival ALG gravitational instanton. These compactify to an isotrivial rational elliptic surface $S$ by addind a finite monodromy fiber $D_{\infty}$ at infinity. By \cite[Section~3.1]{Hein}, $S \setminus D_{\infty}$  deformation retracts onto the \textit{dual} finite monodromy fiber. Therefore the intersection form of $H^2(X,\mathbb{Z})$ is an extended Dynkin diagram.  Next, for example, assume the extended Dynkin diagram is $\tilde D_4$, then $H^2(X,\mathbb{Z})$ is generated by $[E_i]$, $i=1,2,...,5$, with $[E_i]^2=-2$ for all $i$, $[E_i] \cdot [E_j] =1$ for all $\{i,j\}=\{1,2\}, \{1,3\}, \{1,4\}, \{1,5\}$, and $[E_i] \cdot [E_j] =0$ otherwise. The homology class of each fiber is 
\begin{equation} [F]=2[E_1]+[E_2]+[E_3]+[E_4]+[E_5].
\end{equation}
The intersection numbers of $[F]$ with all $[E_i]$ are $0$.
We write $[\Sigma_i]$ as
\begin{equation}
[\Sigma_i] = a_i[F] + b_i[E_1]+c_i[E_2]+d_i[E_3]+e_i[E_4].
\end{equation}
Then the self-intersection number of $b_i[E_1]+c_i[E_2]+d_i[E_3]+e_i[E_4]$ is $-2$. The extended Dynkin diagram restricted to this subset is the unextended Dynkin diagram, which has negative definite intersection form. This implies that there are only finitely many distinct $b_i[E_1]+c_i[E_2]+d_i[E_3]+e_i[E_4]$ with self-intersection $-2$. Then, we use $\int_F \omega_{t_i} = \int_F \omega_{1,0}\not=0$ to control $a_i$. The proof for other extended Dynkin diagrams are similar.

This proves a uniform curvature bound, and this yields a bound on the $\omega_i$-holomorphic radius exactly as in \cite[Theorem~7.4]{CCIII}. The proof of \cite[Theorem~7.4]{CCIII} relies on \cite[Proposition~2.1]{Ruan}, which is valid in the ALG case since these are volume non-collapsed in bounded scales at infinity. Theorem 7.6, Lemma 7.7 and Theorem 7.8 of \cite{CCIII} then go through exactly the same in the $\ALG$ cases, with $U_N$ replaced by balls of radius $1$. Note that only the hyperk\"ahler condition is used in \cite[Theorem~7.6]{CCIII}.
\end{proof}

The equation $\omega_{t_i}^2=\omega_{1,0}^2$ and the bound on $\mathrm{tr}_{\omega_{t_i}}\omega_{1,0}=\mathrm{tr}_{\omega_{1,0}}\omega_{t_i}$ imply that there exists a constant $C$ indepedent of $t_i$ such that 
\begin{equation}
C^{-1}\omega_{1,0}\le \omega_{t_i}\le C \omega_{1,0}.
\end{equation}
Since the difference $\omega_{1,t_i}-\omega_{1,0}$ decays uniformly, there exists a constant $R$ such that $\frac{1}{2}\omega_{1,0}\le \omega_{t_i}\le 2\omega_{1,0}$ for all $s\ge R$. So 
\begin{equation}
|\Delta_{\omega_{1,0}}\varphi_{t_i}| = |\mathrm{tr}_{\omega_{1,0}}(\omega_{t_i}-\omega_{1,t_i})| = |\mathrm{tr}_{\omega_{1,0}}\omega_{t_i}-2-s_{t_i}\Delta_{\omega_{1,0}}(\chi \cdot \log|u|)|\le C
\end{equation}
on $X$. Moreover,
\begin{equation}
\begin{split}
&\int_X|\Delta_{\omega_{1,0}}\varphi_{t_i}|\frac{\omega_{1,0}^2}{2} \\
&\le \int_X(\mathrm{tr}_{\omega_{1,0}}\omega_{t_i}-2) \frac{\omega_{1,0}^2}{2}
 + |s_{t_i}| \int_X|\Delta_{\omega_{1,0}}(\chi \cdot \log|u|)| \frac{\omega_{1,0}^2}{2} \le C,
\end{split}
\end{equation}
where we have used the fact that $\omega_{t_i}^2=\omega_{1,0}^2$, which implies that $\mathrm{tr}_{\omega_{1,0}}\omega_{t_i}\ge 2$.
So for $\delta=\frac{1}{100}$, $||\Delta_{\omega_{1,0}}\varphi_{t_i}||_{L^2_{-1+\delta}(X,\omega_{1,0})}\le C$. Now we consider the operator $\Delta_{\omega_{1,0}}:W^{2,2}_{1+\delta}(X,\omega_{1,0})\to L^2_{-1+\delta}(X,\omega_{1,0})$. By the ALG weighted analysis in \cite{CCI, CCII, CCIII}, and  \cite{HHM},
it is easy to see that any function in the kernel of this operator must be a constant, and consequently there exists another function $\tilde\varphi_{t_i}\in W^{2,2}_{1+\delta}(X,\omega_{1,0})$ such that $\varphi_{t_i}-\tilde\varphi_{t_i}$ is a constant and
\begin{equation}
\begin{split}
||\tilde\varphi_{t_i}||_{W^{2,2}_{1+\delta}(X,\omega_{1,0})} &\le C||\Delta_{\omega_{1,0}}\tilde\varphi_{t_i}||_{L^2_{-1+\delta}(X,\omega_{1,0})} \\
&= C||\Delta_{\omega_{1,0}}\varphi_{t_i}||_{L^2_{-1+\delta}(X,\omega_{1,0})} \le C.
\end{split}
\end{equation}
This implies that $||\tilde\varphi_{t_i}||_{W^{2,p}(\{s\le 4R\},\omega_{1,0})}\le C(p)$ for any $p>1$ using the bound on $|\Delta_{\omega_{1,0}}\varphi_{t_i}|$. For any $\alpha\in(0,1)$, by the Evans-Krylov estimate, 
\begin{equation}
[\partial\bar\partial\tilde\varphi_{t_i}]_{C^\alpha(\{s\le 3R\},\omega_{1,0})}\le C(\alpha);
\end{equation}
see \cite[Section 2.4]{Siu-lectures} for instance.
By standard elliptic estimates, for any $k\in\dN$ and any $\alpha\in(0,1)$, 
\begin{equation}
||\tilde\varphi_{t_i}||_{C^{k,\alpha}(\{s\le 2R\},\omega_{1,0})}\le C(k,\alpha).
\end{equation}

When $s\ge R$,
\begin{equation}
\begin{split}
|\Delta_{\omega_{1,t_i}+\omega_{t_i}}\tilde\varphi_{t_i}| \le C |(\omega_{1,t_i}+\omega_{t_i}) \wedge (\omega_{t_i}-\omega_{1,t_i})|_{\omega_{1,t_i}+\omega_{t_i}} \\
=C |\omega_{1,0}^2-\omega_{1,t_i}^2|_{\omega_{1,t_i}+\omega_{t_i}} \le C |\omega_{1,0}^2-\omega_{1,t_i}^2|_{\omega_{1,0}} \le C s^{-4}.
\end{split}
\end{equation}
Let $\chi_R$ be a cut-off function which is 1 when $s \ge 2R$ and is 0 when $s \le R$. Then 
$|\Delta_{\omega_{1,t_i}+\omega_{t_i}} \xi_{t_i}| \le C s^{-4}$,
where $\xi_{t_i}\equiv \chi_R \cdot \tilde\varphi_{t_i}$

We use the Moser iteration technique to prove that $||\xi_{t_i}||_{C^0}\le C$ for a constant $C$ indepedent of $t_i$ and $p$. For any $j=0,1,2,3,...$ and $p=2^j$,
\begin{equation}
\label{e:integral-by-part}
\begin{split}
&p^2\int_X \xi_{t_i}^{2p-2}|\nabla_{\omega_{1,t_i}+\omega_{t_i}}\xi_{t_i}|^2\frac{(\omega_{1,t_i}+\omega_{t_i})^2}{2}\\
&=\int_X |\nabla_{\omega_{1,t_i}+\omega_{t_i}} (\xi_{t_i}^p)|^2 \frac{(\omega_{1,t_i}+\omega_{t_i})^2}{2}\\
&= - \int_X \xi_{t_i}^p \Delta_{\omega_{1,t_i}+\omega_{t_i}} (\xi_{t_i}^p)\frac{(\omega_{1,t_i}+\omega_{t_i})^2}{2}\\
&= - p(p-1)\int_X \xi_{t_i}^{2p-2} |\nabla_{\omega_{1,t_i}+\omega_{t_i}}\xi_{t_i}|^2\frac{(\omega_{1,t_i}+\omega_{t_i})^2}{2} \\
&- p \int_X \xi_{t_i}^{2p-1}\Delta_{\omega_{1,t_i}+\omega_{t_i}} \xi_{t_i} \frac{(\omega_{1,t_i}+\omega_{t_i})^2}{2}.
\end{split}
\end{equation}
We have used the fact that $\xi_{t_i}-\varphi_{t_i}$ is a constant when $s \ge 2R$, and there exists $\delta_{t_i}>0$ such that $\varphi_{t_i}\in W^{k,2}_{-\delta_{t_i}}(X,\omega_{1,0})$. Therefore, 
\begin{equation}
\begin{split}
\int_X |\nabla_{\omega_{1,t_i}+\omega_{t_i}} (\xi_{t_i}^p)|^2 &\frac{(\omega_{1,t_i}+\omega_{t_i})^2}{2} \\
&= - \frac{p^2}{2p-1} \int_X \xi_{t_i}^{2p-1} \Delta_{\omega_{1,t_i}+\omega_{t_i}} \xi_{t_i} \frac{(\omega_{1,t_i}+\omega_{t_i})^2}{2}.
\end{split}
\end{equation}
Recall that by Theorem 1.2(i) of \cite{Hein-Sobolev}, there exist a constant $C$ and a weight function $\psi$ with $\int_X \psi \frac{\omega_{1,0}^2}{2} = 1$ such that for any $\xi\in C^\infty_0(X)$,
\begin{align}
\label{e:weighted-Sobolev}
\Big(\int_X |\xi-\xi_0|^4 s^{-4} \frac{\omega_{1,0}^2}{2}\Big)^{\frac{1}{2}} \le C \int_X |\nabla_{\omega_{1,0}} \xi|^2 \frac{\omega_{1,0}^2}{2},
\end{align}
where $\xi_0\equiv \int_X \psi \xi \frac{\omega_{1,0}^2}{2}$. Equation~\eqref{e:weighted-Sobolev} also holds for $\xi_{t_i}^p$ because \eqref{e:weighted-Sobolev} is unchanged if we add $\xi$ by a constant and $\xi_{t_i}^p$ can be written as a constant plus a function in $W^{k,2}_{-\delta_{t_i}}(X,\omega_{1,0})$. Then
\begin{equation}
\begin{split}
|\xi_0|^2 &\le C \Big(\int_{s\le 2R} |\xi_0|^4 s^{-4} \frac{\omega_{1,0}^2}{2}\Big)^{\frac{1}{2}}\\
& \le C\Big(\int_X |\xi-\xi_0|^4 s^{-4} \frac{\omega_{1,0}^2}{2}\Big)^{\frac{1}{2}}+C||\xi||_{C^0(\{s\le 2R\}}^2.
\end{split}
\end{equation}
So
\begin{equation}
\Big(\int_X |\xi|^4 s^{-4} \frac{\omega_{1,0}^2}{2}\Big)^{\frac{1}{2}} \le C \int_X |\nabla_{\omega_{1,0}} \xi|^2 \frac{\omega_{1,0}^2}{2}+ C ||\xi||_{C^0(\{s\le 2R\})}^2
\end{equation}
for $\xi=\xi_{t_i}^p$ and a constant $C$ independent of $t_i$ and $p$. Therefore,
\begin{equation}
\begin{split}
\label{e:Moser}
&\Big(\int_X |\xi_{t_i}|^{4p} s^{-4} \frac{\omega_{1,0}^2}{2}\Big)^{\frac{1}{2}}\\
&\le C \int_X |\nabla_{\omega_{1,0}} \xi_{t_i}^p|^2 \frac{\omega_{1,0}^2}{2}+ C ||\xi_{t_i}||_{C^0(\{s\le 2R)\}}^{2p}\\
&\le C \int_X |\nabla_{\omega_{1,t_i}+\omega_{t_i}} (\xi_{t_i}^p)|^2 \frac{(\omega_{1,t_i}+\omega_{t_i})^2}{2}+ C ||\xi_{t_i}||_{C^0(\{s\le 2R)\}}^{2p}\\
&\le \frac{C p^2}{2p-1} \int_X |\xi_{t_i}|^{2p-1} s^{-4} \frac{\omega_{1,0}^2}{2}+ C ||\xi_{t_i}||_{C^0(\{s\le 2R\})}^{2p}\\
&\le \frac{C p^2}{2p-1} \int_X \Big(\frac{2p-1}{2p}|\xi_{t_i}|^{2p}+\frac{1}{2p}\Big) s^{-4} \frac{\omega_{1,0}^2}{2}+ C ||\xi_{t_i}||_{C^0(\{s \le 2R)\}}^{2p}\\
&\le C\Big(p\int_X |\xi_{t_i}|^{2p}  s^{-4} \frac{\omega_{1,0}^2}{2} + 1 + ||\xi_{t_i}||_{C^0(\{s \le 2R)\}}^{2p}\Big).
\end{split}
\end{equation}
For $p=1$, 
\begin{equation}
\begin{split}
&\int_X |\xi_{t_i}|^{2} s^{-4} \frac{\omega_{1,0}^2}{2}\le C\Big(\int_X |\xi_{t_i}|^{4} s^{-4} \frac{\omega_{1,0}^2}{2}\Big)^{\frac{1}{2}}\\
&\le C \int_X |\xi_{t_i}| s^{-4} \frac{\omega_{1,0}^2}{2}+ C ||\xi_{t_i}||_{C^0(\{s \le 2R)\}}^{2}+C\\
&\le C\epsilon \int_X |\xi_{t_i}|^2 s^{-4} \frac{\omega_{1,0}^2}{2} + C\epsilon^{-1} + C ||\xi_{t_i}||_{C^0(\{s\le 2R)\}}^{2}+C
\end{split}
\end{equation}
for all $\epsilon>0$. If we choose $\epsilon$ such that the coefficient $C \epsilon<\frac{1}{2}$, then 
\begin{equation}
\int_X |\xi_{t_i}|^{2} s^{-4} \frac{\omega_{1,0}^2}{2} \le C.
\end{equation}
As in Page 715 of \cite{CCIII}, $||\xi_{t_i}||_{C^0(X)}\le C$ using \eqref{e:Moser}. This implies that 
\begin{equation}
||\varphi_{t_i}||_{C^0}\le ||\tilde\varphi_{t_i}||_{C^0} + |\varphi_{t_i}-\tilde\varphi_{t_i}|\le 2||\tilde\varphi_{t_i}||_{C^0} \le C
\end{equation} 
because $\varphi_{t_i}-\tilde\varphi_{t_i}$ is a constant and $\varphi_{t_i}$ decays. Using the Evans-Krylov estimate and standard elliptic estimates on $B(x,1,\omega_{1,0})$ for any $x\in X$, $||\varphi_{t_i}||_{C^k(X,\omega_{1,0})}\le C(k)$ for constants $C(k)$ indepedent of $t_i$.

The bound on $\int_X |\xi_{t_i}|^{2} s^{-4} \frac{\omega_{1,0}^2}{2}$ also implies a bound on $\int_X |\nabla_{\omega_{1,0}}\xi_{t_i}|^{2} \frac{\omega_{1,0}^2}{2}$ by \eqref{e:integral-by-part}. This implies that 
\begin{equation}
\int_X |\nabla_{\omega_{1,0}}\varphi_{t_i}|^{2} \frac{\omega_{1,0}^2}{2} \le C.
\end{equation} 

Finally, we use \cite[Proposition~2.9]{Hein} to prove that there exist a constant $\delta>0$ and constants $C(k,\delta)>0$ independent of $t_i$ such that
\begin{equation}
\|s^{k+\delta}\nabla^k_{\omega_{1,0}}\varphi_{t_i}\| \le C(k,\delta)
\end{equation}
for all $k$. Then we use the Arzela-Ascoli Lemma, the diagonal arguments, and standard elliptic estimates to finish the proof.

\subsection{Closing Remarks}
There is a folklore conjecture that some examples constructed by Biquard-Boalch \cite{BB} are ALG and by varying parameters, achieve all possible periods satisfying \eqref{nondegcond}. See \cite{FMSW} for some progress towards this conjecture.
We also mention that there is a folklore conjecture that some examples constructed by Biquard-Boalch \cite{BB} and Cherkis-Kapustin \cite{CherkisKapustinALG} are ALG$^*$ and by varying parameters, achieve all possible periods satisfying \eqref{nondegcond}.

\bibliographystyle{amsalpha}

\bibliography{CVZ_II_references}

 \end{document}